\documentclass[11pt, reqno]{amsart}
\usepackage{bbm, bm}
\usepackage[top=2.65cm,left=2.9cm,right=2.9cm,bottom=2.65cm]{geometry}

\usepackage{amsfonts}
\usepackage{amsmath,amsthm,amscd,mathrsfs,amssymb,graphicx,enumerate,
	dsfont}
\usepackage{hyperref}
\usepackage[usenames,dvipsnames]{xcolor}
\usepackage{xypic}
\hypersetup{colorlinks=true,citecolor=NavyBlue,linkcolor=Brown,urlcolor=Orange}

\RequirePackage[l2tabu, orthodox]{nag}
\RequirePackage{fixltx2e}

\usepackage{stmaryrd}
\usepackage{amsmath}
\usepackage{amsfonts}
\usepackage{amsthm}
\usepackage{mathtools}
\usepackage{txfonts}
\usepackage{hyperref}
\usepackage{graphicx,amssymb,amsfonts,amsmath,amscd}
\usepackage{textcomp}
\usepackage[all,ps,cmtip]{xy}
\usepackage{amscd}
\usepackage{xspace}
\usepackage{mathrsfs}
\usepackage{ dsfont }
\usepackage[varg]{pxfonts}

\usepackage{enumerate}
\usepackage{xcolor}
\usepackage{tikz}\usetikzlibrary{decorations.markings,matrix,arrows}
\usepackage{extarrows}

\newtheorem{thm2}{Theorem}
\newtheorem{cor2}{Corollary}
\newtheorem{conj2}{Conjecture}

\newtheorem{prop2}{Proposition}
\newtheorem{thm}{Theorem}[section]
\newtheorem{cor}[thm]{Corollary}
\newtheorem{lem}[thm]{Lemma}
\newtheorem{prop}[thm]{Proposition}
\newtheorem{conj}[thm]{Conjecture}
\theoremstyle{definition}
\newtheorem{defn}[thm]{Definition}

\newtheorem{rmk}[thm]{Remark}
\newtheorem*{rmk2}{Remark}
\newtheorem{ques}[thm]{Question}

\DeclareMathOperator{\ch}{ch}

\DeclareMathOperator{\End}{End} 
\DeclareMathOperator{\Hom}{Hom}

\DeclareMathOperator{\Sym}{Sym}

\DeclareMathOperator{\prim}{prim}

\DeclareMathOperator{\st}{st}

\newcommand{\C}{\ensuremath\mathds{C}}
\newcommand{\R}{\ensuremath\mathrm{R}}
\newcommand{\N}{\ensuremath\mathds{N}}
\newcommand{\Z}{\ensuremath\mathds{Z}}
\newcommand{\Q}{\ensuremath\mathds{Q}}

\newcommand{\FF}{\ensuremath\mathcal{F}}
\newcommand{\h}{\ensuremath\mathfrak{h}}

\newcommand{\PP}{\ensuremath\mathds{P}}

\newcommand{\D}{\ensuremath\mathrm{D}}
\newcommand{\HH}{\ensuremath\mathrm{H}}

\newcommand{\CH}{\ensuremath\mathrm{CH}}

\newcommand{\id}{\ensuremath\mathrm{id}}
\newcommand{\Gr}{\ensuremath\mathrm{Gr}}
\newcommand{\Stab}{\ensuremath\mathrm{Stab}}
\renewcommand{\top}{\ensuremath\mathrm{top}}

\newcommand{\QQ}{\mathbb{Q}}
\newcommand{\NN}{\mathbb{N}}

\newcommand{\Ss}{\mathcal S}

\newcommand{\XX}{\mathcal X}
\newcommand{\YY}{\mathcal Y}
\newcommand{\MM}{\mathcal M}

\newcommand{\OO}{\mathcal O}

\newcommand{\codim}{\hbox{codim}}

\newcommand{\wt}{\widetilde}
\newcommand{\ima}{\hbox{Im}}

\newcommand{\rom}{\romannumeral}
\renewcommand{\1}{\mathop{\mathds{1}}\nolimits} 

\renewcommand{\bar}{\overline}

\newcommand{\cart}{\ar@{}[dr]|\square} 
\renewcommand{\tilde}{\widetilde}

\AtBeginDocument{%
	\def\MR#1{}
}

\newif\ifHideFoot
\HideFoottrue  
\HideFootfalse  

\ifHideFoot

\newcommand{\Lie}[1]{}
\newcommand{\Charles}[1]{}
\newcommand{\Robert}[1]{}

\else 

\newcommand{\marg}[1]{\normalsize{{
			\color{red}\footnote{{\color{blue}#1}}}{\marginpar[\vskip
			-.25cm{\color{red}\hfill$\Rightarrow$\tiny\thefootnote}]{\vskip
				-.2cm{\color{red}$\Leftarrow$\tiny\thefootnote}}}}}
\newcommand{\Lie}[1]{\marg{(Lie) #1}}
\newcommand{\Charles}[1]{\marg{(Charles) #1}}
\newcommand{\Robert}[1]{\marg{(Robert) #1}}

\fi

\begin{document}

	\title[The generalized Franchetta conjecture for some HK varieties, II]{The generalized Franchetta conjecture\linebreak
		 for some hyper-K\"ahler varieties, II.}

	\author{Lie Fu}
	\address{Institut Camille Jordan, Universit\'e Claude Bernard Lyon 1, France}
	\address{IMAPP, Radboud University, Nijmegen, Netherlands}
	\email{fu@math.univ-lyon1.fr}
	
	\author{Robert Laterveer} 
	\address{CNRS - IRMA, Universit\'e de Strasbourg,
		France}
	\email{laterv@math.unistra.fr}  
	
	\author{Charles Vial}
	\address{Fakult\"at f\"ur Mathematik, Universit\"at Bielefeld, Germany} 
	\email{vial@math.uni-bielefeld.de}
	
	\thanks{LF is supported by the Radboud Excellence Initiative programme.
	LF and RL are supported by the project FanoHK, grant ANR-20-CE40-0023 of the French Agence Nationale de la Recherche.}

	
	\begin{abstract} 
		We prove the generalized Franchetta conjecture for the locally complete family of hyper-K\"ahler  eightfolds constructed by Lehn--Lehn--Sorger--van Straten (LLSS). As a
		corollary, we establish the Beauville--Voisin conjecture for very general
		LLSS eightfolds. 
		The strategy consists in 
		reducing to the Franchetta property for
		 relative fourth powers of cubic fourfolds, by 
		using the recent description of LLSS eightfolds as moduli spaces of Bridgeland semistable objects in the Kuznetsov component of the derived category of cubic fourfolds, together with its generalization to the relative setting  due to Bayer--Lahoz--Macr\`i--Nuer--Perry--Stellari.
		As a by-product, we compute the Chow motive of the Fano variety of lines on a smooth cubic hypersurface in terms of the Chow motive of the cubic hypersurface.
	\end{abstract}
	
	\maketitle
	
		\vspace{-20pt}
	\section*{Introduction}

	\subsection*{The Franchetta property}
	Let $f: \mathcal X \to B$ be a smooth projective morphism between smooth schemes of finite type over the field of complex numbers. 
	For any fiber $X$ of $f$ over a closed point of~$B$, we define 
	$$\operatorname{GDCH}^*_B(X) := \operatorname{Im} \bigl(\CH^*(\mathcal X) \to \CH^*(X)\bigr),$$ the image of the Gysin restriction map. Here and in the sequel, Chow groups are always considered with \emph{rational} coefficients. 
	The elements of $\operatorname{GDCH}^*_B(X)$ are called the \emph{generically defined cycles} (with respect to the deformation family $B$) on 
	$X$.
	The morphism $f: \mathcal X \to B$ is said to satisfy the \emph{Franchetta property} for codimension-$i$ cycles if the restriction of the cycle class map 
	$$\operatorname{GDCH}^i_B(X) \to \HH^{2i}(X,\QQ)$$ is injective for all (or equivalently, for very general) fibers $X$. It is said to satisfy the Franchetta property if it satisfies the Franchetta property for codimension-$i$ cycles for all $i$. At this point, we note that if $B'$ is a smooth locally closed subscheme of $B$, then there is no \emph{a priori} implication between the Franchetta properties for $\XX \to B$ and for the restricted family $\XX_{B'}\to B'$\,: informally, $\operatorname{GDCH}^i_{B'}(X)$ is generated by more elements than $\operatorname{GDCH}^i_B(X)$\,; on the other hand, specializing to $B'$ creates new relations among cycles. However, if $B' \to B$ is a dominant morphism, the Franchetta property for $\XX_{B'} \to B'$ implies the Franchetta property for $\XX \to B$\,; see \cite[Remark~2.6]{FLV}.

The Franchetta property is a property about the generic fiber $\mathcal{X}_{\eta}$. Indeed it is equivalent to the condition that the composition 
$$\CH^{*}(\mathcal{X}_{\eta})\to\CH^{*}(\mathcal{X}_{\bar\eta})\to \HH^{*}(\mathcal{X}_{\bar\eta})$$ is injective, where the first map, which is always injective, is the pull-back to the geometric generic fiber and the second one is the cycle class map to some Weil cohomology of $\mathcal{X}_{\bar \eta}$.

	\subsection*{Hyper-K\"ahler varieties}
	It was first conjectured by O'Grady~\cite{OGradyK3} that the universal family of polarized K3 surfaces of given genus $g$ over the corresponding moduli space satisfies the Franchetta property. By using Mukai models, this was proved for certain families of K3 surfaces of low genus by Pavic--Shen--Yin~\cite{FranchettaK3}. By investigating the case of the Beauville--Donagi family~\cite{BD} of Fano varieties of lines on smooth cubic fourfolds, we were led in~\cite{FLV} to ask whether O'Grady's conjecture holds more generally for hyper-K\"ahler varieties\,:

	\begin{conj2}[Generalized Franchetta conjecture for hyper-K\"ahler varieties \cite{FLV}]
Let $\FF$ be the moduli stack of a locally
complete family of polarized hyper-K\"ahler varieties. Then the universal family $\XX \to \FF$ satisfies the Franchetta property.
	\end{conj2}
	
It might furthermore be the case that, for some positive integers $m$, the relative $m$-th powers $\XX^m_{/\FF} \to \FF$ satisfy the Franchetta property. This was proved for instance in the case $m=2$ in \cite{FLV} for the universal family of K3 surfaces of genus $\leq 12$ (but different from $11$) and for the Beauville--Donagi family of Fano varieties of lines on smooth cubic fourfolds.\medskip

	The first main object of study of this paper is about the locally complete family of hyper-K\"ahler eightfolds
	constructed by Lehn--Lehn--Sorger--van Straten~\cite{LLSvS}, subsequently
	referred to as LLSS eightfolds. An LLSS eightfold is constructed from the space of twisted cubic curves on a smooth cubic fourfold not containing a plane.
	The following result, which is the first main result of this paper, completes our previous work \cite[Theorem 1.11]{FLV} where the Franchetta property was established for 0-cycles and codimension-2 cycles on LLSS eightfolds.
	
	\begin{thm2}\label{thm2:FranchettaLLSS}
	The universal family of LLSS hyper-K\"ahler eightfolds over the moduli space of smooth cubic fourfolds not containing a plane satisfies the Franchetta property. 
	\end{thm2}

As already observed in \cite[Proposition~2.5]{FLV}, the generalized Franchetta conjecture for a family of hyper-K\"ahler varieties implies the Beauville--Voisin conjecture~\cite{V17} for the very general member of the family\,:

\begin{cor2} Let $Z$ be an LLSS hyper-K\"ahler eightfold.
	Then the $\QQ$-subalgebra $$\mathrm{R}^*(Z) := \langle h, c_j(Z) \rangle \quad \subset \CH^*(Z)$$ generated by the natural polarization $h$ and the Chern classes $c_j(Z)$ injects into cohomology \emph{via} the cycle class map. In particular,
the very general LLSS eightfold satisfies the Beauville--Voisin conjecture~\cite{V17}.  
\end{cor2}

Indeed, any subring of $\CH^*(Z)$ generated by generically defined cycles injects in cohomology, provided $Z$ satisfies the Franchetta property. As such, one
 further obtains as a direct consequence of Theorem~\ref{thm2:FranchettaLLSS} that for an LLSS eightfold $Z$,  the subring of $\CH^*(Z)$ generated by  the polarization $h$, the Chern classes $c_j(Z)$ and the classes of the (generically defined) co-isotropic subvarieties described in \cite[Corollary~1.12]{FLV} injects in cohomology \emph{via} the cycle class map. This provides new evidence for Voisin's refinement in \cite{Voisin2016} of the Beauville--Voisin conjecture. 
	
\subsection*{Strategy of proof of Theorem~\ref{thm2:FranchettaLLSS}} 
 In \cite[Theorem~1.11]{FLV}, we established the Franchetta conjecture for 0-cycles on LLSS eightfolds. Our proof consisted in reducing, \emph{via} Voisin's degree-6 dominant rational map $\psi : F\times F \dashrightarrow Z$ constructed in~\cite[Proposition~4.8]{Voisin2016}, to the generalized Franchetta conjecture for
  the square of the Fano variety of lines $F$ of a smooth cubic fourfold, which we established in \cite[Theorem~1.10]{FLV}.
In order to deal with positive-dimensional cycles on $Z$, we take a completely different approach\,: we consider the recent description of LLSS eightfolds as certain moduli spaces of semistable objects in the Kuznetsov component of the derived category of cubic fourfolds \cite{LLMS, LPZ18}, together with its  generalization due to Bayer--Lahoz--Macr\`i--Nuer--Perry--Stellari \cite{BLMNPS} to the relative setting. 
Our first task, which is carried out in \S \ref{S:motive-moduli}, consists then in relating the Chow motives of the moduli space $M$ of semistable objects in the Kuznetsov  component of the derived category of a smooth cubic fourfold $Y$ to the Chow motives of powers of $Y$. 
By adapting and refining an argument of B\"ulles~\cite{Bue18}, we show in Theorem~\ref{thm:NCK3} that the motive of $M$ belongs to the thick subcategory generated by Tate twists of the motive of $Y^m$, where $\dim M = 2m$. 
 Since all the data involved in the above are generically defined, the Franchetta property for LLSS eightfolds is thus reduced to the Franchetta property for fourth powers of smooth cubic fourfolds (Theorem~\ref{thm2:Fcubic} below). 
 The proof of Theorem~\ref{thm2:FranchettaLLSS} is then given in \S \ref{S:proof}\,; see Theorem~\ref{thm:FranchettaLLSS}.

 \subsection*{Powers of smooth cubic hypersurfaces} 

 The following theorem, in the case of cubic fourfolds, suggests that the Franchetta property could hold for powers of Fano varieties of cohomological K3-type\,; that (conjectural) motivic properties of hyper-K\"ahler varieties could transfer to Fano varieties of cohomological K3-type was already pinpointed in~\cite{FLV-MCK}.
 
 \begin{thm2}\label{thm2:Fcubic}
Let $B$ be the open subset of $\PP \HH^0(\PP^{n+1},\mathcal{O}(3))$ parameterizing smooth cubic hypersurfaces in $\PP^{n+1}$, and let $\mathcal Y \to B$ be the corresponding universal family. Then the families $\mathcal Y^m_{/B} \to B$ satisfy the Franchetta property for all $m\leq 4$.
 \end{thm2}

 Theorem~\ref{thm2:Fcubic} is established in \S \ref{SS:cubic}. Its proof relies on
   the existence of a multiplicative Chow--K\"unneth decomposition for cubic hypersurfaces  (see Theorem~\ref{thm:MCKcubic}), and on an analogue in the case of cubic hypersurfaces
 of a result of Yin~\cite{YinK3} concerning K3 surfaces (which itself is analogous to a result of Tavakol~\cite{Tavakol} concerning hyperelliptic curves). The latter is embodied in Corollary~\ref{cor:TautoCubic}. In the particular case of cubic \emph{fourfolds}, it admits also the following refined form, which is the analogue of Voisin's \cite[Conjecture 1.6]{V17}\,:

 	\begin{prop2}[see Proposition~\ref{proppowers}]
 		 Let $Y$ be a smooth cubic fourfold and
 		$m\in\N$. Let $\widetilde{\mathrm{R}}^{\ast}(Y^m)$ be the $\QQ$-subalgebra
 		\[ \tilde{\mathrm{R}}^\ast(Y^m):=\langle p_i^\ast \CH^1(Y), \ p_j^\ast \CH^2(Y), \
 		p_{k,l}^\ast
 		\Delta_Y
 		\rangle\ \ \ \subset\ \CH^\ast(Y^m),\ \]
 		where $p_i$, $p_j$ and $p_{k,l}$ denote the various projections from $Y^m$ to
 		$Y$
 		and $Y^2$.
 		Then $\tilde{\mathrm{R}}^*(Y^m)$ injects into $\HH^{2*}(Y^m,\QQ)$ \emph{via} the cycle class
 		map for all $m \leq 2b_{\mathrm{tr}}(Y)+1$, where $b_{\mathrm{tr}}(Y)$ denotes the dimension of the transcendental cohomology of the smooth cubic fourfold $Y$.
 	Moreover, $\tilde{\mathrm{R}}^*(Y^m)$ injects into $\HH^{2*}(Y^m,\QQ)$ for all $m$ if and only
 	if $Y$ is Kimura--O'Sullivan finite-dimensional~\cite{KimuraFD}.
 \end{prop2}

\subsection*{Fano varieties of lines on smooth cubic hypersurfaces} 
Combining Theorem~\ref{thm2:Fcubic} with our previous work \cite[Theorem~4.2]{FLV-MCK} 
where we established the Franchetta property for the square of the Fano variety of lines on a smooth cubic hypersurface, we can compute explicitly  the Chow motive of the Fano variety of lines on a smooth cubic hypersurface in terms of the Chow motive of the cubic hypersurface, without resorting to Kimura--O'Sullivan finite-dimensionality arguments. The following is the main result\,; see Theorem~\ref{thm:FY2} for more precise statement and stronger results.

 \begin{thm2}\label{thm2:FY}
 Let $Y$ be a smooth cubic hypersurface in $\PP^{n+1}$ and $F$ the associated Fano variety of lines on~$Y$. We have an isomorphism of Chow motives
\begin{equation}\label{eqn:GalkinShinderMot}
\h(F)\simeq \Sym^{2}(\h^{n}(Y)_{\prim}(1))\oplus \bigoplus_{i=1}^{n-1}\h^{n}(Y)_{\prim}(2-i)\oplus \bigoplus_{k=0}^{2n-4}\1(-k)^{\oplus a_{k}},
\end{equation} 
where 
\begin{equation*}
a_{k}=\begin{cases}
\lfloor \frac{k+2}{2}\rfloor &\text{ if } k<n-2 \\
\lfloor \frac{n-2}{2}\rfloor &\text{ if } k=n-2\\
\lfloor \frac{2n-2-k}{2}\rfloor &\text{ if } k>n-2.\\
\end{cases}
\end{equation*} 
In particular, we have an isomorphism of Chow motives 
\begin{equation}\label{eqn:QuestionHuy}
\h(F)(-2) \oplus \h(Y) \oplus \h(Y)(-n) \simeq \mathrm{Sym}^2\h(Y).
\end{equation}
\end{thm2}
 
 \begin{rmk2}
 Let us explain the relations of Theorem \ref{thm2:FY} to earlier works and open questions in the literature. 
 \begin{enumerate}[$(i)$]
 \item The isomorphism \eqref{eqn:GalkinShinderMot} lifts the isomorphism in cohomology due to Galkin--Shinder \cite[Theorem 6.1]{GalkinShinder} to the level of Chow motives.
 \item  The isomorphism \eqref{eqn:QuestionHuy} answers a question of Huybrechts \cite[\S3.3]{Huy}.
\item Theorem~\ref{thm2:FY} refines the main result of~\cite{Lat-Y-FY} and, in fact, our method of proof, which goes through the Franchetta property for $F\times F$ established in \cite[Theorem~4.2]{FLV-MCK} and a cancellation property described in Proposition~\ref{P:cancellation}, provides a new, independent, and more conceptual, proof of the main result of~\cite{Lat-Y-FY}.
\item Specializing to the case of cubic fourfolds, \eqref{eqn:GalkinShinderMot} implies 
an isomorphism of Chow motives
\begin{equation}\label{eqn:FanoCubicFourfold}
\h(F) \simeq \mathrm{Sym}^2M \oplus M(-1),
\end{equation} 
where  $M:=  (Y,\Delta_Y - \frac{1}{3}h^4\times Y -  \frac{1}{3}h^2\times h^2 - \frac{1}{3}Y\times h^4,1)= \mathds{1} \oplus \h^4(Y)_{\mathrm{prim}}(1) \oplus \mathds{1}(-2) $ is the \emph{K3-surface-like} Chow motive. 
Recall that for a smooth projective surface $S$, the Hilbert square of $S$ is isomorphic to the blow-up of the symmetric square $\mathrm{Sym}^2 S$ along the diagonal and that $\h(\mathrm{Hilb}^2(S)) \simeq \mathrm{Sym}^2\h(S) \oplus \h(S)(-1)$. Therefore, \eqref{eqn:FanoCubicFourfold} can be interpreted  as saying that the Chow motive of $F$ is the Hilbert square of the Chow motive  $M$ of a ``non-commutative'' K3 surface\,; this is the motivic analogue of the following folklore conjecture (\emph{cf.}\ \cite[Conjecture~4.3]{Pop})\,: 
\emph{given a smooth cubic fourfold $Y$, the derived category of the Fano variety of lines $F$ is equivalent, as a $\C$-linear triangulated category, to the symmetric square 
(in the sense of Ganter--Kapranov \cite{MR3177367}) of the Kuznetsov component $\mathcal{A}_{Y}$ of the derived category of $Y$, \emph{i.e.},
$\D^{b}(F)\cong \mathrm{Sym}^{2}\mathcal{A}_{Y}.$}

 \end{enumerate}
 \end{rmk2}

\subsection*{Further outlooks} The strategy for proving Theorem~\ref{thm2:FranchettaLLSS} has potential beyond the case of LLSS eightfolds. Indeed, once suitable stability conditions are constructed for other non-commutative K3 surfaces (\emph{cf}.\ \S \ref{S:motive-moduli}), one may hope our strategy can be employed to prove the generalized Franchetta conjecture for the associated hyper-K\"ahler varieties.
In \S \ref{S:moduli}, we exemplify the above by establishing the Franchetta property for many (non locally complete) families of hyper-K\"ahler varieties.

\subsection*{Conventions} All algebraic varieties are defined over the field of complex numbers. We work with Chow groups with rational coefficients. The categories of motives we consider are the categories of pure Chow motives with rational coefficients $\MM_{\mathrm{rat}}$ and of pure numerical motives with rational coefficients $\MM_{\mathrm{num}}$,  as reviewed in \cite{Andre}. We write $\h(X)$ for the Chow motive of a smooth projective variety $X$. The set of non-negative integers will be denoted by $\N$.

\subsection*{Acknowledgements} We thank Thorsten Beckmann, Tim-Henrik B\"ulles, Chunyi Li and Xiaolei Zhao for helpful exchanges. We thank the referees for constructive comments that improved our paper.

	\section{The motive of moduli spaces of objects in 2-Calabi--Yau categories}
	\label{S:motive-moduli}
	An important source of examples of projective hyper-K\"ahler manifolds is given by moduli spaces of stable sheaves on Calabi--Yau surfaces \cite{Muk84} \cite{MR1703077} \cite{MR1966024} \cite{Yos01}. Recently, B\"ulles \cite{Bue18} showed that the Chow motive of such a (smooth projective) moduli space is in the thick tensor subcategory generated by the motive of the surface. By allowing non-commutative ``Calabi--Yau surfaces'', we get even more examples of hyper-K\"ahler varieties as moduli spaces of stable objects in a 2-Calabi--Yau category equipped with stability conditions \cite{BM14a} \cite{BM14b} \cite{BLMS} \cite{BLMNPS} \cite{PPZ19}. B\"ulles' result was recently extended to this non-commutative setting by Floccari--Fu--Zhang \cite[Theorem 5.3]{FFZ19}. In this section, we provide a refinement of these results following an observation of Laterveer \cite{LatLagFib}.\medskip
		
	Let $Y$ be a smooth projective variety and let $\D^{b}(Y)$ be its bounded derived category of coherent sheaves. Let  $\mathcal{A}$ be an admissible triangulated subcategory of $\D^{b}(Y)$ such that it is \emph{2-Calabi--Yau}, that is, its Serre functor is the double shift $[2]$. There are by now several interesting examples of such 2-Calabi--Yau categories\,: 
	\begin{enumerate}[$(i)$]
	\item  the twisted derived categories of K3 or abelian surfaces equipped with a Brauer class\,;
	\item the Kuznetsov component of the derived categories of cubic fourfolds \cite{Kuz10}\,;
	\item the Kuznetsov component of the derived categories of Gushel--Mukai fourfolds or sixfolds \cite{Gus83} \cite{Muk89} \cite{DK19} \cite{KP16}\,;
	\item the Kuznetsov component of the derived categories of  Debarre--Voisin twentyfolds \cite{MR2746467}.
	\end{enumerate}
	In all the examples listed above, there is a semi-orthogonal decomposition $$\D^{b}(Y)=\langle\mathcal{A}, {}^{\perp}\mathcal{A}\rangle,$$ where ${}^{\perp}\mathcal{A}:=\{E\in \D^{b}(Y)~\vert~\Hom(E, F)=0 \text{ for all } F\in \mathcal{A}\}$ is generated by an exceptional collection (with $ {}^{\perp}\mathcal{A} = 0$ in $(i)$)\,; see \cite{MS18} or \cite[Example 5.1]{FFZ19} for more details.  
	
 Let us now proceed to review the notions of Mukai lattice and Mukai vector. For that purpose, recall that the topological K-theory of $Y$ is naturally equipped with the Euler pairing\,:
	$$([E], [F]):=-\chi(E, F), \ \text{ for all } [E], [F]\in K_{\top}(Y).$$
	Following \cite{AT14}, the \emph{Mukai lattice} of $\mathcal{A}$ is defined as the free abelian group
		$$\HH(\mathcal{A}):=\{\alpha \in K_{\top}(Y)~\vert~(\alpha, [E])=0 \text{ for all } E\in {}^{\perp}\mathcal{A}\}/\text{torsion},$$
	equipped with the restriction of the Euler pairing, which is called the \emph{Mukai pairing}. The \emph{Mukai vector} of an object $E\in \mathcal{A}$ is by definition $v(E):=\ch(E)\cdot \sqrt{\operatorname{td}_{Y}}$. It is an element of $\HH(\mathcal{A})$ by construction.

	Assume that $\mathcal{A}$ admits stability conditions in the sense of Bridgeland \cite{Bri07}\,; this has by now been established  for examples $(i) \sim (iii)$ \cite{Bri08} \cite{YY14} \cite{BLMS} \cite{PPZ19},
	 and is also expected for example $(iv)$. We denote the distinguished connected component of the stability manifold by $\Stab^\dagger(\mathcal{A})$. Recall that if $v$ is a primitive element in the Mukai lattice of $\mathcal{A}$, a stability condition $\sigma\in \Stab^\dagger(\mathcal{A})$ is said to be $v$-\emph{generic} if stability coincides with semi-stability for all objects in $\mathcal{A}$ with Mukai vector~$v$. General results in \cite{Lieblich-JAG06} and \cite{AHLH-Existence} guarantee that a good moduli space $\mathcal{M}_{\sigma}(\mathcal{A}, v)$ of $\sigma$-semistable objects in $\mathcal{A}$ with Mukai vector $v$ exists as an algebraic space of finite type. 
	 Moreover, initiated by Bayer--Macr\`i \cite{BM14a, BM14b}, a much deeper study shows that the moduli space $\mathcal{M}_{\sigma}(\mathcal{A}, v)$ is a non-empty projective hyper-K\"ahler manifold for examples  $(i)$, $(ii)$ and $(iii)$,  by \cite{BM14b}, \cite{BLMNPS}, and \cite{PPZ19} respectively (the example $(iv)$ is also expected).

	\begin{thm}\label{thm:NCK3}
Let $Y$ be a smooth projective variety and $\mathcal{A}$ be an admissible triangulated subcategory of $\D^b(Y)$ such that $\mathcal{A}$ is 2-Calabi--Yau. Let $v$ be a primitive element in the Mukai lattice of $\mathcal{A}$ and let $\sigma\in \Stab^\dagger(\mathcal{A})$ be a $v$-generic
stability condition. Assume that $\mathcal{M}:=\mathcal{M}_{\sigma}(\mathcal{A}, v)$ is a non-empty smooth projective hyper-K\"ahler variety of dimension $2m:=v^{2}+2$. Then the Chow motive $\h(\mathcal{M})$ is a direct summand of a Chow motive of the form $$\bigoplus_{i=1}^{r}\h(Y^m)(l_{i})$$ with $r\in \N$, $l_{i}\in \Z$.
\end{thm}
	 The novelty of this result with respect to \cite[Theorem 5.3]{FFZ19} is the better bound on the power of $Y$, which will be crucial in the proof of the Franchetta property. 
	
	\begin{proof}[Proof of Theorem \ref{thm:NCK3}]
	Following B\"ulles \cite{Bue18}, we consider the following chain of two-sided ideals of the ring of self-correspondences of $\mathcal{M}$\,:
	$$I_{0}\subset I_{1}\subset \dots \subset \CH^{*}(\mathcal{M}\times \mathcal{M}),$$
	where for any non-negative integer $k$, $$I_{k}:=\langle \beta\circ \alpha~\vert~\alpha\in \CH^{*}(\mathcal{M}\times Y^{k}), \beta\in\CH^{*}(Y^{k}\times \mathcal{M})\rangle.$$
	Note that $I_{0}=\langle \alpha\times \beta~\vert~ \alpha, \beta\in \CH^{*}(M)\rangle$ consists of ``decomposable'' cycles in $\mathcal{M}\times \mathcal{M}$. The conclusion of the theorem can be rephrased as saying that $\Delta_{\mathcal{M}}\in I_{m}$.
	
	Using Lieberman's formula, B\"ulles showed \cite[Theorem 1.1]{Bue18} that the intersection product behaves well with respect to the grading. More precisely, for any $k, k'\geq 0$,
	\begin{equation}\label{eqn:multiplicativity}
	 I_{k}\cdot I_{k'}\subset I_{k+k'}.
	\end{equation}
	The observation of Laterveer \cite[Lemma 2.2]{LatLagFib} is that the vanishing of the irregularity of $\mathcal{M}$ implies that any divisor of $\mathcal{M}\times \mathcal{M}$ is decomposable, that is, 
	\begin{equation}\label{eqn:divisor}
	 \CH^{1}(\mathcal{M}\times \mathcal{M})\subset I_{0}.
	\end{equation}
	It was pointed out in \cite[Proposition 5.2]{FFZ19} that the proof of Markman's result \cite[Theorem 1]{Mar02} (revisited in \cite{MZ17})  goes through for any 2-Calabi--Yau category, and we have  $$\Delta_{\mathcal{M}} = c_{2m}(\mathcal{P}) \in \CH^{2m}(\mathcal{M} \times \mathcal{M}), $$
	with $\mathcal{P}:=-R\pi_{1,3,*}(\pi_{1,2}^{*}(\mathcal{E})^{\vee}\otimes^{\mathbb{L}}\pi_{2,3}^{*}(\mathcal{E}))$, where $\mathcal{E}$ is a universal family and the $\pi_{i,j}$'s are the natural projections from $\mathcal{M}\times S\times \mathcal{M}$. Therefore our goal is to show that $c_{2m}(\mathcal{P})\in I_{m}$. We prove by induction that $c_{i}(\mathcal{P})\in I_{\lfloor i/2\rfloor}$ for any $i\in \N$.
	
	The cases $i=0$ and 1 are clear from \eqref{eqn:divisor}. For $i\geq 2$, as in \cite{Bue18}, the Grothendieck--Riemann--Roch theorem implies that $\ch(\mathcal{P}) = - (\pi_{1,3})_\ast(\pi_{1,2}^\ast\alpha \cdot \pi_{2,3}^\ast\beta) =\beta\circ \alpha$,   where
    $$ \alpha = \ch(\mathcal{E}^\vee) \cdot \pi_Y^\ast\sqrt{\operatorname{td}(Y)} \quad\text{ in } \CH^*(\mathcal{M}\times Y)\quad\text{and}\quad \beta = \ch(\mathcal{E}) \cdot \pi_Y^\ast\sqrt{\operatorname{td}(Y)} \quad\text{ in } \CH^*(Y\times \mathcal{M})\ .$$ 
       Hence $\ch_{i}(\mathcal{P})\in I_{1}$ for all $i\in \N$, by definition. Let us drop $\mathcal{P}$ from the notation in the sequel. Note that $$\ch_{i}=\frac{(-1)^{i-1}}{(i-1)!}c_{i}+Q(c_{1},\dots, c_{i-1}),$$ where $Q$ is a weighted homogeneous polynomial of degree $i$. It suffices to show that $$Q(c_{1}, \dots, c_{i-1})\in I_{\lfloor i/2\rfloor}.$$ To this end, for any monomial  $c_{1}^{d_{1}}c_{2}^{d_{2}}\cdots c_{i-1}^{d_{i-1}}$ of $Q$, we have that $\sum_{j=1}^{i-1}jd_{j}=i$ and hence 
    \begin{equation}\label{eqn:index}
     \sum_{j=1}^{i-1}\lfloor j/2\rfloor d_{j}\leq \lfloor i/2\rfloor.
    \end{equation}
     Using the induction hypothesis that $c_{j}\in I_{\lfloor j/2\rfloor}$ for any $j\leq i-1$, we see that 
    $$c_{1}^{d_{1}}c_{2}^{d_{2}}\cdots c_{i-1}^{d_{i-1}}\in \prod_{j=1}^{i-1}I^{d_{j}}_{\lfloor j/2\rfloor}\subseteq I_{\sum_{j=1}^{i-1}\lfloor j/2\rfloor d_{j}}\subseteq I_{\lfloor i/2\rfloor},$$
    where $\prod$ denotes the intersection product, the second inclusion uses the multiplicativity \eqref{eqn:multiplicativity}, and the last inclusion follows from \eqref{eqn:index}.	The induction process is complete. We conclude that $\Delta_{\mathcal{M}} = c_{2m}$ belongs to $I_{m}$, as desired.
	\end{proof}

	In Theorem \ref{thm:NCK3}, if the Mukai vector $v$ is not primitive, the moduli space of semistable objects $\mathcal{M}_{\sigma}(\mathcal{A}, v)$ is no longer smooth (for any stability condition $\sigma$).  
	However,  in the so-called O'Grady-10 case, namely $v=2v_{0}$ with $v_{0}^{2}=2$, for a generic stability condition $\sigma$,  there exists a crepant resolution of $\mathcal{M}_{\sigma}(\mathcal{A}, v)$, which is a projective hyper-K\"ahler tenfold.  
	In the classical case of moduli spaces of $H$-semistable sheaves with such a Mukai vector $v$ on a K3 or abelian surface with the polarization $H$ being $v$-generic, such a crepant resolution was constructed first by O'Grady \cite{MR1703077} for some special  $v$, and then by Lehn--Sorger \cite{LehnSorger-OG10} and Perego--Rapagnetta \cite{PR13} in general. 
	In the broader setting where $\mathcal{A}=\D^b(S)$ is the derived category of a K3 surface and $\sigma$ is a $v$-generic stability condition, the existence of a crepant resolution of $\mathcal{M}_{\sigma}(\mathcal{A}, v)$ was proved by Meachan--Zhang \cite[Proposition 2.2]{MeachanZhang} using \cite{BM14b}.
	In our general setting where $\mathcal{A}$ is a 2-Calabi--Yau category, Li--Pertusi--Zhao~\cite[Section 3]{LPZ19}  showed that the singularity  of $\mathcal{M}_{\sigma}(\mathcal{A}, v)$ has the same local model as in the classical case and that the construction of the crepant resolution in \cite{LehnSorger-OG10} can be adapted by using \cite{AHR-Annals}. For details we refer to \cite[Section 3]{LPZ19}, where the proof was written for cubic fourfolds, but works in general.

	The following two results exemplify the belief that the Chow motive of the crepant resolution can be controlled in the same way as in Theorem \ref{thm:NCK3}. 
	Theorem~\ref{thm:K3OG} is for  K3 and abelian surfaces, while Theorem~\ref{thm:NCK3OG} is in the non-commutative setting  where $\mathcal{A}$ is the Kuznetsov component of a cubic fourfold.
	
	\begin{thm}\label{thm:K3OG}
	Let $S$ be a K3 or abelian surface and let $\alpha$ be a Brauer class of $S$. Let $v_{0}\in \widetilde{\HH}(S)$ be a Mukai vector with $v_{0}^{2}=2$. Set $v=2v_0$. Let $\sigma$ be a $v$-generic stability condition on $\D^{b}(S, \alpha)$. Denote by $\widetilde{\mathcal{M}}$ any crepant resolution of the moduli space $\mathcal{M}:=\mathcal{M}_{\sigma}(S, v)$ of $\sigma$-semistable objects in $\D^{b}(S, \alpha)$ with Mukai vector $v$. Then the Chow motive $\h(\widetilde{\mathcal{M}})$ is a direct summand of a Chow motive of the form $$\bigoplus_{i=1}^{r}\h(S^{5})(l_{i})$$ with $r\in \N$, $l_{i}\in \Z$.
	\end{thm}
	The existence of crepant resolutions for $\mathcal{M}:=\mathcal{M}_{\sigma}(S, v)$ was established in \cite[Proposition~2.2]{MeachanZhang}.
	The fact that  $\h(\widetilde{\mathcal{M}})$ is in the tensor subcategory generated by $\h(S)$ is proved in \cite[Theorem~1.3]{FFZ19}. The improvement here consists in bounding the power of $S$ by 5.
	\begin{proof}
	Replacing $\mathcal{M}$ everywhere by the stable locus $\mathcal{M}^{\st}$ in the proof of Theorem \ref{thm:NCK3} yields that 
	\begin{equation}\label{eqn:DeltaMstable}
	 \Delta_{\mathcal{M}^{\st}}\in \langle \beta\circ \alpha~\vert~\alpha\in \CH^{*}(\mathcal{M}^{\st}\times S^{5}), \beta\in\CH^{*}(S^{5}\times \mathcal{M}^{\st})\rangle.
	\end{equation}
Indeed, let us give some details\,:
 define the chain of subgroups $I_{0}\subset I_{1}\subset \dots \subset \CH^{*}(\mathcal{M}^{\st}\times \mathcal{M}^{\st})$ by $I_{k}:=\langle \beta\circ \alpha~\vert~\alpha\in \CH^{*}(\mathcal{M}^{\st}\times S^{k}), \beta\in\CH^{*}(S^{k}\times \mathcal{M}^{\st})\rangle$. 
 Note that, since $\mathcal{M}^{\st}$ is not proper, $\CH^{*}(\mathcal{M}^{\st}\times \mathcal{M}^{\st})$ is no longer a ring for the composition of self-correspondences. It is however easy to see that the multiplicativity \eqref{eqn:multiplicativity} and the inclusion \eqref{eqn:divisor} still hold. Again by \cite[Proposition 5.2]{FFZ19}, $\Delta_{\mathcal{M}^{\st}}=c_{10}(\mathcal{P})$ with $\mathcal{P}$ defined similarly as in the proof of Theorem \ref{thm:NCK3}, by using the universal sheaf over $\mathcal{M}^{\st}$. The Grothendieck--Riemann--Roch theorem implies that the Chern characters of $\mathcal{P}$ belong to $I_{1}$. The same induction process as in the proof of Theorem~\ref{thm:NCK3} shows that the $i$-th Chern class of $\mathcal{P}$ lies in $I_{\lfloor i/2\rfloor}$ for all $i$. In particular,  $\Delta_{\mathcal{M}^{\st}}=c_{10}(\mathcal{P})\in I_{5}$, which is nothing but \eqref{eqn:DeltaMstable}.

The rest of the proof is as in \cite[Section 4]{FFZ19}. Let us give a sketch. First, as two birationally isomorphic hyper-K\"ahler varieties have isomorphic Chow motives \cite{Riess-Manuscripta}, it is enough to consider the crepant resolution constructed by O'Grady and Lehn--Sorger. By construction, there is a further blow-up $\widehat{\mathcal{M}}\to \widetilde{\mathcal{M}}$ whose boundary $\partial \widehat{\mathcal{M}}:=\widehat{\mathcal{M}}\backslash \mathcal{M}^{\st}$ is the union of two divisors denoted by $\widehat{\Omega}$ and $\widehat{\Sigma}$. By taking closures, \eqref{eqn:DeltaMstable} implies that there exist $\widehat{\alpha_{i}}\in \CH^{*}(\widehat{\mathcal{M}}\times S^{5}), \widehat{\beta_{i}}\in\CH^{*}(S^{5}\times \widehat{\mathcal{M}})$ such that $
\Delta_{\widehat{\mathcal{M}}}-\sum_{i}\widehat{\beta_{i}}\circ \widehat{\alpha_{i}}$ is supported on $\partial\widehat{\mathcal{M}}\times \widehat{\mathcal{M}}\cup \widehat{\mathcal{M}}\times \partial \widehat{\mathcal{M}}$. Consequently, there is a split injection of Chow motives
$$\h(\widehat{\mathcal{M}})\hookrightarrow \bigoplus_{i} \h(S^{5})(l_{i})\oplus \h(\widehat{\Omega}) \oplus \h(\widehat{\Omega})(-1)\oplus \h(\widehat{\Sigma})\oplus \h(\widehat{\Sigma})(-1).$$
It remains to show that $\h(\widehat{\Omega})$ and $\h(\widehat{\Sigma})$ are both direct summands of Chow motives of the form $\bigoplus_{i=1}^{r}\h(S^{5})(l_{i})$. 

For $\h(\widehat{\Omega})$,  the proof of \cite[Lemma~4.3]{FFZ19} shows that $\h(\widehat{\Omega})$ has a split injection into a Chow motive of the form $\bigoplus_{i}\h(\mathcal{M}_{\sigma}(S, v_{0}))(l_{i})$. One can conclude by Theorem~\ref{thm:NCK3} that $\h(\mathcal{M}_{\sigma}(S, v_{0}))$ is a direct summand of a Chow motive of the form $\bigoplus_{i} \h(S^{2})(l_{i})$. 

Finally, for $\h(\widehat{\Sigma})$, the proof of \cite[Lemmas~4.2 and~4.3]{FFZ19} shows that $\h(\widehat{\Sigma})$ has a split injection into a Chow motive of the form $\bigoplus_{i}\h(\mathcal{M}_{\sigma}(S, v_{0})^{2})(l_{i})$. Again by Theorem \ref{thm:NCK3}, it is a direct summand of a Chow motive of the form $\bigoplus_{i} \h(S^{4})(l_{i})$.
 	\end{proof}

	\begin{thm}\label{thm:NCK3OG}
	Let $Y$ be a smooth cubic fourfold and let $$\mathcal{A}_{Y}\colon\!\!\!\!=\langle \mathcal{O}_Y,\mathcal{O}_Y(1),\mathcal{O}_Y(2)\rangle^{\perp}=\{E\in \D^b(Y)~\mid~ \operatorname{Ext}^*(\mathcal{O}_Y(i), E)=0 \text{ for } i=0, 1, 2\}$$ be its Kuznetsov component. Let $v_{0}$ be an element in the Mukai lattice of $\mathcal{A}_{Y}$ with $v_{0}^{2}=2$ and set $v=2v_{0}$. Let $\sigma\in \Stab^\dagger(\mathcal{A}_{Y})$ be a $v$-generic stability condition. Let $\mathcal{M}:=\mathcal{M}_{\sigma}(\mathcal{A}_Y, v)$ be the moduli space of $\sigma$-semistable objects in $\mathcal{A}_{Y}$ with Mukai vector $v$. Let $\widetilde{\mathcal{M}}$ be a crepant resolution of $\mathcal{M}$, which is a projective hyper-K\"ahler manifold of dimension~$10$. Then its Chow motive $\h(\widetilde{\mathcal{M}})$ is a direct summand of a Chow motive of the form $$\bigoplus_{i=1}^{r}\h(Y^5)(l_{i})$$ with $r\in \N$, $l_{i}\in \Z$.
	\end{thm}
	The existence of crepant resolutions for  $\mathcal{M}:=\mathcal{M}_{\sigma}(\mathcal{A}_Y, v)$ was proved in \cite[Theorem~3.1]{LPZ19}.
The novelty of the above theorem compared to \cite[Theorem 1.8]{FFZ19} consists in bounding the power of $Y$ by 5.
	\begin{proof}
	Using the arguments of \cite[Section~5]{FFZ19}, one sees that the proof of the theorem is the same as that of Theorem~\ref{thm:K3OG} by replacing $S$ by $Y$.
	\end{proof}

\begin{rmk}
	By way of example, let $V=V(Y)$ be the compactification of the ``twisted intermediate jacobian filtration'' constructed in \cite{voisintwisted}\,; this $V$ is a hyper-K\"ahler tenfold of O'Grady-10 type. There exists a tenfold $\wt{\mathcal{M}}$ as in Theorem \ref{thm:NCK3OG} that is birational to $V$ \cite[Theorem 1.3]{LPZ19}, and so one obtains a split injection of Chow motives
  \[ \h(V)\simeq \h(\wt{\mathcal{M}})\ \hookrightarrow\ \bigoplus_{i=1}^{r}\h(Y^5)(l_{i})   \] 
  (here the isomorphism follows from \cite{Riess-Manuscripta}). In particular, this implies that $V$ satisfies the standard conjectures, by Arapura \cite[Lemma 4.2]{Arapura06}.
\end{rmk}

	\section{The Franchetta property for fourth powers of cubic hypersurfaces} \label{S:cubic}
	For a morphism $\mathcal Y \to B$ to a smooth scheme $B$ of finite
	type over a field and for $Y$ a fiber over a closed point of $B$, we define, for
	all positive integers~$m$, 
\begin{equation}\label{E:Fm}
\operatorname{GDCH}_B^*(Y^m) := \operatorname{Im}\big(\CH^*(\mathcal{Y}_{/B}^m) \to \CH^*(Y^m) \big),
\end{equation}
 where
	$\CH^*(\mathcal{Y}_{/B}^m) \to \CH^*(Y^m)$ is the Gysin restriction map. Focusing on smooth projective
	complex morphisms $\mathcal Y \to B$, we say that $\mathcal Y \to B$ (or by
	abuse $Y$, if the family it fits in is clear from the context) satisfies the
	\emph{Franchetta property} for $m$-th powers if the restriction of
	the cycle class map $\operatorname{GDCH}_B^*(Y^m) \to \HH^*(Y^m, \QQ)$ is injective.
	
	Our aim is to establish the Franchetta property for fourth powers
	of cubic hypersurfaces\,; see Theorem \ref{thm2:Fcubic} in the introduction.

	\subsection{Generically defined cycles and tautological cycles} We adapt the stratification argument \cite[Proposition 5.7]{FLV} (which was for Mukai models of K3 surfaces) to its natural generality.
	We first record the following standard fact.
	\begin{lem}\label{trivial}
	Let $P$ be a smooth projective variety. The following conditions are equivalent\,:
	\begin{enumerate}[$(i)$]
	\item The Chow motive of $P$ is of Lefschetz type\,:
	$$\h(P)\simeq \bigoplus_{i=1}^{r}\1(l_{i}),$$
	for some integers $r\geq 1$, $l_{1}, \dots, l_{r}$. 
	\item The cycle class map $\CH^{*}(P)\to \HH^{*}(P, \QQ)$ is an isomorphism.
	\item The cycle class map $\CH^{*}(P)\to \HH^{*}(P, \QQ)$ is injective.
	\item The Chow group $\CH^{*}(P)$ is a finite-dimensional $\QQ$-vector space.
	\item The Chow groups of powers of $P$ satisfy the K\"unneth formula: for any $m\in \NN$, $\CH^{*}(P^{m})\cong \CH^{*}(P)^{\otimes m}$.
	\end{enumerate}
	\end{lem}
	\begin{proof}
	The implications $(i)\Longrightarrow(ii)\Longrightarrow(iii)\Longrightarrow(iv)$ are obvious, while the implication
	$(iv)\Longrightarrow(i)$ is in \cite{KimuraRep} and \cite{VialRep}.
The implication
	$(i)+(ii)\Longrightarrow(v)$ is also clear by the K\"unneth formula for cohomology. It remains to show $(v)\Longrightarrow(iv)$. Suppose $\CH^{*}(P^{2})\cong \CH^{*}(P)\otimes \CH^{*}(P)$. Then there exist $\alpha_{i}, \beta_{i}\in \CH^{*}(P)$ such that $$\Delta_{P}=\sum_{i=1}^{r}\alpha_{i}\times \beta_{i}.$$ In particular, the identity morphism of $\CH^{*}(P)$ factors through an $r$-dimensional $\QQ$-vector space. Therefore $\CH^{*}(P)$ is finite dimensional. 
	\end{proof}
	\begin{defn} \label{def:trivialChow}
	We say a smooth projective variety $P$ has \emph{trivial} Chow groups if $P$ satisfies one of the equivalent conditions in Lemma \ref{trivial}.
	\end{defn}
	Examples of varieties with trivial Chow groups include homogeneous varieties, toric varieties, and varieties whose bounded derived category of coherent sheaves admits a full exceptional collection \cite{MR3331729}. Conjecturally, a smooth projective complex variety has trivial Chow groups if and only if its Hodge numbers $h^{i,j}$ vanish for all $i\neq j$.

	\begin{defn}[Tautological rings]\label{def:tauto}
	 Let $P$ be a smooth projective variety with trivial Chow groups and $Y$ a smooth subvariety. The \emph{tautological ring} of $Y$ is by definition the $\QQ$-subalgebra
	$$\mathrm{R}^{*}(Y):=\langle \operatorname{Im}(\CH^{*}(P)\to \CH^{*}(Y)), c_{i}(T_{Y}) \rangle \subset\ \CH^{*}(Y)$$
	generated by the restrictions of cycles of $P$ and the Chern classes of the tangent bundle~$T_{Y}$. Note that if $Y$ is the zero locus of a dimensionally transverse section of a vector bundle on $P$, the Chern classes of $T_{Y}$ automatically come from $P$.
	More generally, for any $m\in\N$, we define the {\em tautological ring} of $Y^{m}$ as the $\QQ$-subalgebra
	  \[ \mathrm{R}^\ast(Y^m):=\langle  p_i^*\mathrm{R}^{*}(Y), p_{j,k}^*\Delta_Y\rangle\ \ \subset\ \CH^\ast(Y^m)\]
	  generated by pull-backs of tautological classes on factors and pull-backs of the diagonal $\Delta_Y\subset Y\times Y$. Here, $p_i$ and $p_{j,k}$ denote the various projections from $Y^m$ to
	  $Y$ and to $Y^2$. Note that by Lemma \ref{trivial} $(v)$, the cycles coming from the ambient space are all tautological\,: $\operatorname{Im}(\CH^{*}(P^{m})\to \CH^{*}(Y^{m}))\subset \mathrm{R}^{*}(Y^{m})$.
	  \end{defn}
	  Similar subrings are studied for hyperelliptic curves by Tavakol \cite{Tavakol}, for K3 surfaces  by Voisin \cite{V17} and Yin \cite{YinK3}, and for cubic hypersurfaces by Diaz \cite{Diaz}.
	  
	  Given an equivalence relation $\sim$ on $\{1, \dots, m\}$, we define the corresponding \emph{partial diagonal} of $Y^{m}$ by $\{(y_{1}, \dots, y_{m})\in Y^{m}~\mid~ y_{i}=y_{j} \text{ if } i
	  \sim j\}$. 
	  Natural projections and inclusions along partial diagonals between powers of $Y$ preserve the tautological rings. More generally, we have the following fact, which implies that the system of tautological cycles in Definition \ref{def:tauto} is the smallest one that is preserved by natural functorialities and contains Chern classes and cycles restricted from the ambient spaces.
	  \begin{lem}[Functoriality]\label{preserved} Notation is as before.	
	Let $\phi :I\to J$ be a map between two finite sets and let $f: Y^{J}\to Y^{I}$ be the corresponding morphism. Then
	  \[  f_\ast \mathrm{R}^\ast(Y^{J})\ \ \subset\ \mathrm{R}^\ast(Y^I)\quad  \text{ and } \quad   f^{*} \mathrm{R}^\ast(Y^{I})\ \ \subset\ \mathrm{R}^\ast(Y^J)\ .\] 
	  \end{lem}
	  
	  \begin{proof}
	  The fact that tautological rings are preserved by $f^{*}$ is clear from the definition. Let us show that they are preserved by $f_{*}$. By writing $\phi$ as a composition of a surjective map and an injective map, it is enough to show the lemma in these two cases separately.
	  
	  When $\phi$ is surjective, $f: Y^{J}\hookrightarrow Y^{I}$ is a partial diagonal embedding. Choosing a section of $\phi$ gives rise to a projection $p: Y^{I}\to Y^{J}$ such that $p\circ f=\id_{Y^{J}}$. For any $\alpha\in \mathrm{R}^{*}(Y^{J})$, we have 
	  $$f_{*}(\alpha)=f_{*}(f^{*}(p^{*}(\alpha)))=p^{*}(\alpha)\cdot f_{*}(1_{Y^{J}})$$
	  by the projection formula, where $1_{Y^{J}}$ is the fundamental class of $Y^{J}$. It is clear that $p^{*}(\alpha)$ and $f_{*}(1_{Y^{J}})$ are both in $\mathrm{R}^{*}(Y^{I})$.
	  
	  We leave to the reader the proof in the case where $\phi$ is injective ($f: Y^{J}\twoheadrightarrow Y^{I}$ is then a projection), which is not needed later.
	    \end{proof}

	\begin{defn} [Condition $(\star_r)$\,; {\cite[Definition~5.6]{FLV}}]\label{defn:star} Let $E$ be a vector bundle on a variety~$P$. Given an integer $r\in\N$, we say that the pair $(P, E)$ satisfies condition $(\star_r)$ if for any $r$ distinct points
$x_1,\ldots,x_r\in P$, the evaluation map
  $$\HH^0(P, E) \rightarrow  \bigoplus_{i=1}^r E({x_i})$$
  is surjective, where $E(x)$ denotes the fiber of $E$ at $x$\,; or equivalently, $\HH^{0}(P, E\otimes I_{x_{1}}\otimes\dots\otimes I_{x_{r}})$ is of codimension $r\cdot \operatorname{rank}(E)$ in $\HH^{0}(P, E)$. Clearly, $(\star_r)$ implies $(\star_k)$ for all $k<r$. Note that condition $(\star_1)$  is exactly the global generation of $E$.
\end{defn}			
		
\begin{prop}[Generic \emph{vs.}~Tautological]\label{strat}
Let $P$ be a smooth projective variety with trivial Chow groups and $E$ a globally generated vector bundle on $P$. Denote $\bar{B}:=\PP \HH^0(P, E)$. Let $B$ be a Zariski open subset of $\bar{B}$ parameterizing smooth zero loci of sections of $E$ of dimension $\dim(P)-\operatorname{rank}(E)$. Let $\YY\to \bar{B}$ be the universal family. Assume $(P, E)$ satisfies condition $(\star_r)$. Then
  \[ \operatorname{GDCH}^\ast_B(Y^r) = \mathrm{R}^\ast(Y^r) \]
  for any fiber $Y$ of $\YY\to \bar{B}$ over a closed point of $B$.	Here, $ \operatorname{GDCH}^\ast_B(Y^r)$ is as in~\eqref{E:Fm}.
\end{prop}	  

\begin{proof} This is an adaptation of \cite[Proposition 5.7]{FLV}, which relies on the notion of {\em stratified projective bundle}~\cite[Definition 5.1]{FLV}. Let $q\colon \YY^{r}_{/\bar{B}}\to P^r$ be the natural projection. The morphism $q$ is a stratified projective bundle, where the strata of $P^r$ are defined by the different types of incidence relations of $r$
points in $P$\,:
  \[ \xymatrix{
\YY_{m}=\YY\ar@{^{(}->}[r]\ar[d]^{q_{m}}\cart& \ \cdots \  \ar@{^{(}->}[r] \cart& \YY_{1}\ar@{^{(}->}[r] \ar[d]^{q_{1}} \cart& \YY_{0}=\YY^{r}_{/\bar{B}} \ar[d]^{q_{0}=q} \ar[r]& \bar{B}\\
T_{m}=P \ar@{^{(}->}[r]& \ \cdots \ \ar@{^{(}->}[r]& T_{1}\ar@{^{(}->}[r]& T_{0}=P^{r}&
.}\]
By definition, this means that $q_i : \YY_i \backslash \YY_{i+1} \to T_i\backslash T_{i+1}$, the restriction of $q$ to $T_i\backslash T_{i+1}$, is a projective bundle.
Let us write $\YY_i^\prime$ for the Zariski closure of $\YY_i\setminus\YY_{i+1}$.
Let $Y_i$ (resp. $Y$) denote the fiber of the morphism $\YY_i^\prime\to \bar{B}$ (resp. $\YY\to \bar{B}$) over a closed point $b\in B$, hence $Y_{0}=Y^{r}$. Let $\iota_i\colon Y_i\hookrightarrow Y^r$ denote the natural inclusion.
An application of \cite[Proposition 5.2]{FLV} (which holds for any stratified projective bundle) gives that the generically defined cycles $\operatorname{GDCH}^\ast_B(Y^r):=\ima\bigl(\CH^\ast(\YY^{r}_{/\bar{B}})\to \CH^\ast(Y^r)\bigr)$ can be expressed as follows\,:
  \begin{equation}\label{stratacontribute} \operatorname{GDCH}^\ast_B(Y^r)=\sum_{i=0}^{m}(\iota_{i})_\ast\ima\left({q_{i}|_{Y_{i}}}^{*}: \CH^\ast(T_{i})\to \CH^\ast(Y_{i})\right)\ .
  \end{equation} 
 We proceed to show inductively (just as in \cite[Proof of Proposition 5.7]{FLV}) that each term on the right-hand side of (\ref{stratacontribute}) is in the tautological ring $\mathrm{R}^\ast(Y^r)$\,:

 \noindent $\bullet$ For $i=0$, this follows simply from the fact that $\CH^{*}(T_0)=\CH^{*}(P^r)\cong\CH^{*}(P)^{\otimes r}$.

\noindent $\bullet$ Assume a general point of $T_{i}$ parameterizes $r$ points of $P$ with at least two of them coinciding. Then the contribution of the $i$-th summand of (\ref{stratacontribute}) factors through $\operatorname{GDCH}^{*}_B(Y^{r-1})$ \emph{via} the diagonal push-forward. By the induction hypothesis, this is contained in the diagonal push-forward of $\mathrm{R}^\ast(Y^{r-1})$, hence by Lemma~\ref{preserved} is contained in $\mathrm{R}^{*}(Y^{r})$.

\noindent $\bullet$ Assume a general point of $T_{i}$ parameterizes $r$ distinct points of $P$. In that case, the condition $(\star_r)$ guarantees that the codimension of
 $\YY^\prime_i$ in  $\YY_{i-1}$ is equal to $\codim_{T_{i-1}} (T_{i})$.
  The excess intersection formula (\cite[\S 6.3]{MR1644323}), applied to the cartesian square
\begin{equation*}
\xymatrix{
\YY_{i}=\YY_{i+1}\cup \YY^\prime_{i}\ar@{^{(}->}[r] \ar[d]  \cart& \YY_{i-1}\ar[d]\\
T_{i} \ar@{^{(}->}[r]& T_{i-1}\ ,
}
\end{equation*}
tells us that modulo the $(i+1)$-th term of \eqref{stratacontribute}, the contribution of the $i$-th term is contained in the $(i-1)$-th term. 

\noindent $\bullet$ Finally, the contribution of the $i=m$ term of (\ref{stratacontribute}) is the push-forward of $\mathrm{R}^{*}(Y)$, \emph{via} the small diagonal embedding $Y\hookrightarrow Y^{m}$. This is contained in $\mathrm{R}^{*}(Y^{m})$ by Lemma~\ref{preserved}.
\end{proof}

	\subsection{Multiplicative Chow--K\"unneth decomposition for smooth cubic hypersurfaces}\label{SS:MCK}
	Let  $Y\subset \PP^{n+1}$ be a smooth hypersurface of degree $d$. Recall that if $h := c_1(\mathcal{O}_{\PP^{n+1}}(1)|_{Y}) \in \CH^1(Y)$ denotes the hyperplane section, then the correspondences
		\begin{equation}\label{eq:CKabsolute}
	\pi^{2i} := \frac{1}{d} h^{n-i} \times h^i, \quad 2i\neq n, 	\quad	
	\text{and} \quad \pi^n := \Delta_{Y} - 	\sum_{2i\neq n} \pi^{2i}
	\end{equation}
	in $\CH^n(Y\times Y)$ define a Chow--K\"unneth decomposition, \emph{i.e.}, a decomposition of the Chow motive of $Y$ as a direct sum of summands $\h^i(Y):=(Y,\pi^i)$ with cohomology of pure weight $i$\,; see \emph{e.g.}~\cite[Appendix~C]{MNP}. 
	\medskip
	
The following theorem is stated explicitly in~\cite[Theorem~4.4]{FLV-MCK}, although it can be deduced from a result of~\cite{Diaz} (see \cite[Remark~4.13]{FLV-MCK})\,:

	\begin{thm}\label{thm:MCKcubic} 
	Suppose $Y$ is a smooth cubic $n$-fold.
Then the Chow--K\"unneth decomposition \eqref{eq:CKabsolute} is
multiplicative\,; \emph{i.e.}, the cycles $\pi^k \circ \delta_Y \circ (\pi^i \otimes \pi^j)$ vanish in $\CH^{2n}(Y\times Y \times Y)$ for all $k\neq i+j$.
	\end{thm}
	
	The notion of \emph{multiplicative Chow--K\"unneth (MCK) decomposition} was introduced by Shen--Vial~\cite{SV}.
	 While the existence of a Chow--K\"unneth decomposition is expected for all smooth projective varieties --- in which case the cycles $\pi^k \circ \delta_Y \circ (\pi^i \otimes \pi^j)$ vanish in $\HH^{4n}(Y\times Y \times Y)$ for all $k\neq i+j$ due to the fact that the cohomology ring $\HH^*(Y,\QQ)$ is graded, there are examples of varieties that do not admit such a \emph{multiplicative} decomposition\,; it was however conjectured~\cite[Conjecture~4]{SV}, following the seminal work of Beauville and Voisin \cite{BVK3}, that all hyper-K\"ahler varieties admit such a decomposition. We refer to \cite{FLV-MCK} for a list of hyper-K\"ahler varieties for which the conjecture has been established, as well as for some evidence that Fano varieties of cohomological K3 type (\emph{e.g.}\ smooth cubic fourfolds) should also admit such a decomposition, but also for examples of varieties not admitting such a decomposition.
	
	A new proof of Theorem~\ref{thm:MCKcubic} is given in \cite{FLV-MCK}. The strategy consists in reducing to the Franchetta property for
	the universal family $\FF\to B$ 
	of Fano varieties of lines in smooth cubic hypersurfaces and its relative square. Note that, in that reduction step, we in fact
	established the Franchetta property for $\YY\to B$ and for $\YY\times_B \YY \to
	B$. \medskip
	
	That the canonical Chow--K\"unneth decomposition~\eqref{eq:CKabsolute} of a hypersurface be multiplicative can be spelled out explicitly as follows. (Note that the expression \eqref{E:mck} with $d=3$ corrects the expression for the cycle~$\gamma_3$ in \cite[\S 2]{Diaz}.)
	
	\begin{prop}[MCK relation]\label{prop:MCKcubic2}
		Let $Y$ be a smooth  hypersurface in $\PP^{n+1}$ of degree $d$ and let $h:=
		c_1(\mathcal{O}_{\PP^{n+1}}(1)|_Y)$. Then the Chow--K\"unneth decomposition~\eqref{eq:CKabsolute} is multiplicative if and only if we have the following identity in $\CH^{2n}(Y\times Y\times Y)$, which will be subsequently referred to as the \emph{MCK relation}\,:
		\begin{align}\label{E:mck}
		\delta_Y = & \  \frac{1}{d} \big( p_{1,2}^*\Delta_Y\cdot p_3^*h^n +
		p_{1,3}^*\Delta_Y\cdot p_2^*h^n + p_{2,3}^*\Delta_Y \cdot p_1^*h^n \big) \nonumber \\
		& - \frac{1}{d^2} \big( p_{1}^*h^n\cdot p_2^*h^n + p_{1}^*h^n\cdot p_3^*h^n 
		+p_{2}^*h^n\cdot p_3^*h^n \big) \\
		& + \frac{1}{d^2} \sum_{\substack{i+j+k = 2n\\0<i,j,k<n}} p_{1}^*h^i\cdot p_2^*h^j\cdot p_3^*h^k,\nonumber
		\end{align}
		where $p_i$ and $p_{j,k}$ denote the various projections from $Y^3$ to
		$Y$
		and to $Y^2$.
	\end{prop}
	\begin{proof} 
		 Let us define $\pi^{2i}_{\mathrm{alg}} := \frac{1}{d} h^{n-i}\times h^i$ for $0 \leq i \leq n$\,; these coincide with $\pi^{2i}$ if $2i\neq n$. A direct calculation shows that 
		 	\begin{equation}\label{E:algMCK}
		\pi^{2k}_{\mathrm{alg}}\circ \delta_Y \circ (\pi^{2i}_{\mathrm{alg}} \otimes \pi^{2j}_{\mathrm{alg}}) := \left\{
		 \begin{array}{ll}
		0 \qquad \qquad \qquad \qquad \qquad \text{if} \ k\neq i+j\\
		\frac{1}{d^2} \,  p_1^*h^{n-i}\cdot  p_2^*h^{n-j} \cdot p_3^*h^{k}\quad \, \text{if} \ k=i+j.
		 \end{array}
		 \right.
		 \end{equation}
		 In addition, if $k\neq i+j$, we have
		 $$\pi^{2k}_{\mathrm{alg}}\circ \delta_Y \circ (\pi^{2i}_{\mathrm{alg}} \otimes \pi^{2j}_{\mathrm{alg}}) = \Delta_Y \circ \delta_Y \circ(\pi^{2i}_{\mathrm{alg}} \otimes \pi^{2j}_{\mathrm{alg}}) = \pi^{2k}_{\mathrm{alg}}\circ \delta_Y \circ (\Delta_Y \otimes \pi^{2j}_{\mathrm{alg}})= \pi^{2k}_{\mathrm{alg}}\circ \delta_Y \circ (\pi^{2i}_{\mathrm{alg}} \otimes \Delta_Y).$$ 
		Assume temporarily that $Y$ has no primitive cohomology, \emph{i.e.}, $Y$ has degree 1 or $Y$ is an odd-dimensional quadric. In that case,  we have $\Delta_Y = \sum_i \pi^{2i}_{\mathrm{alg}}$\,; in particular by~\eqref{E:algMCK} the Chow--K\"unneth decomposition~\eqref{eq:CKabsolute} is multiplicative. Furthermore, from $$\Delta_Y =  \sum_i \pi^{2i}_{\mathrm{alg}} = \frac{1}{d}\sum_{i+j=n} p_1^*h^i\cdot p_2^*h^j,$$ we obtain $$\delta_Y = p_{1,2}^*\Delta_Y \cdot p_{1,3}^*\Delta_Y = \frac{1}{d^2}\, \sum_{i+j+k=2n} p_1^*h^i\cdot p_2^*h^j\cdot p_3^*h^k.$$ Combining the above two expressions for $\Delta_Y$ and $\delta_Y$ yields the relation~\eqref{E:mck}.
	
	\noindent	 From now on, we therefore assume that $Y$ has non-trivial primitive cohomology. In that case, the above relations imply that the Chow--K\"unneth decomposition~\eqref{eq:CKabsolute} is multiplicative if and only if 	$$\pi^{2n}\circ \delta_Y \circ (\pi^n \otimes \pi^n) =  \delta_Y \circ (\pi^n
		\otimes \pi^n).$$
		Substituting $\pi^n = \Delta_Y - \sum_{i\neq n} \pi^i$ into the identity
		$\pi^{2n}\circ \delta_Y \circ (\pi^n \otimes \pi^n) =  \delta_Y \circ (\pi^n
		\otimes \pi^n)$, and developing, yields an identity of the form 
		$$\delta_Y = \lambda\, \big( p_{1,2}^*\Delta_Y\cdot p_3^*h^n +
		p_{1,3}^*\Delta_Y\cdot p_2^*h^n + p_{2,3}^*\Delta_Y \cdot p_1^*h^n \big) +
		Q(p_1^*h, p_2^*h,p_3^*h)$$ 
		in $\CH^{2n}(Y\times Y\times Y)$, where $\lambda$ is a rational number to be
		determined and $Q$ is a symmetric rational polynomial in 3 variables to be
		determined. Projecting on the first two factors yields an identity in
		$\CH^n(Y\times Y)$ 
		$$\Delta_Y = d\lambda\, \Delta_Y + R(p_1^*h,p_2^*h)$$ for some symmetric
		rational polynomial $R$ in 2 variables. Since $\Delta_Y$ is not of the form
		$S(p_1^*h,p_2^*h)$ for some symmetric rational polynomial $S$ in 2 variables
		(otherwise, the cohomology ring of $Y$ would be generated by $h$),
		 we find that
		$\lambda = 1/d$.
		The coefficients of the polynomial $Q$ are then obtained by successively applying $(p_{1,2})_*((-)\cdot p_3^*h^{n-k})$ for $k\geq 0$ and by
		symmetrizing.
		Note that in the above we use the identity
		\begin{equation}\label{eqn:RelationX2}
		\Delta_Y\cdot p_1^*h = \Delta_Y\cdot p_2^*h= \frac{1}{d} \sum_{i+j=n+1}
		p_1^*h^i \cdot p_2^*h^j,
		\end{equation} 
		which is obtained by applying the excess intersection formula  \cite[Theorem
		6.3]{MR1644323} 
		to the following cartesian diagram,
		with excess normal bundle $\mathcal{O}_{Y}(d)$,
		\begin{equation*}
		\xymatrix{
			Y \ar[d]\ar[r]^-{\Delta_{Y}}& Y\times Y\ar[d]\\
			\PP^{n+1} \ar[r]& \PP^{n+1}\times \PP^{n+1}
		}
		\end{equation*}
		together with the relation $\Delta_{\PP^{n+1}}=\sum_{i+j=n+1}p_1^*h^i \cdot
		p_2^*h^j$ in
		$\CH^{n+1}(\PP^{n+1}\times \PP^{n+1})$, where by abuse we have denoted $h$ a
		hyperplane section of both $\PP^{n+1}$ and $Y$.
	\end{proof}
	
	\begin{rmk} \label{R:MCK}
		Using the above-mentioned fact that the cycles $\pi^k \circ \delta_Y \circ (\pi^i \otimes \pi^j)$ vanish in $\HH^{4n}(Y\times Y \times Y)$ for all $k\neq i+j$,
		the proof of Proposition~\ref{prop:MCKcubic2} also establishes that the MCK relation~\eqref{E:mck} holds in $\HH^{4n}(Y\times Y\times Y,\QQ)$ for all smooth hypersurfaces $Y$ in $\PP^{n+1}$. 
	\end{rmk}

	\subsection{On the tautological ring of smooth cubic hypersurfaces}
	\label{SS:cohotauto}
	We first consider a smooth hypersurface $Y\subset \PP^{n+1}$ of any degree $d>0$.
	In what follows, $b_{\mathrm{pr}}(Y)$ denotes the dimension of the middle primitive cohomology 
	$\HH^n(Y,\QQ)_{\prim}$
	of $Y$, \emph{i.e.}, the dimension of the orthogonal complement in $\HH^*(Y, \QQ)$ of the subalgebra generated by the hyperplane section. In this paragraph we are interested in understanding the intersection theory of \emph{tautological cycles} on powers of~$Y$\,:
	
		\begin{ques}\label{conjpowers2} 
  Let $(Y,h)$ be a polarized smooth projective variety\,; \emph{i.e.}~a smooth projective variety equipped with the class $h:= c_1(\mathcal L) \in \CH^1(Y)$ of an ample line bundle $\mathcal L$ on $Y$.
	For $m\in\N$, denote ${\mathrm{R}}^\ast(Y^m)$ the $\QQ$-subalgebra
\[ {\mathrm{R}}^\ast(Y^m):=\langle p_i^\ast h, \ 
p_{k,l}^\ast
\Delta_Y
\rangle\ \ \ \subset\ \CH^\ast(Y^m),\ \]
where $p_i$ and $p_{k,l}$ denote the various projections from $Y^m$ to
$Y$
and to $Y^2$. (In case $Y$ is a hypersurface in $\PP^{n+1}$, this is the tautological ring of $Y^m$ in the sense of Definition~\ref{def:tauto}).
For which integers
$m\in \N$ does ${\mathrm{R}}^*(Y^m)$ inject into $\HH^{2*}(Y^m,\QQ)$ \emph{via} the cycle class
map\,?
	\end{ques}

We note that if $(Y,h)$ satisfies the conclusion of Question~\ref{conjpowers2} for $m=2$, then $\Delta_Y\cdot p_1^*h^{\dim Y}$ must be a rational multiple of $p_1^*h^{\dim Y} \cdot p_2^*h^{\dim Y}$ in $\CH^{2\dim Y}(Y\times Y)$. This fails for a very general curve of genus $\geq 4$ where $h$ is taken to be the canonical divisor\,; see \cite{GG} and \cite{Yin-FP}.  Furthermore, the MCK relation~\eqref{E:mck} does not hold for a very general curve of genus~$\geq 3$ (see \cite[Proposition~7.2]{MR4028066})
so that the conclusion of Question~\ref{conjpowers2} for $m=3$ fails  for a very general curve of genus~$\geq 3$. Note however that Tavakol~\cite{Tavakol2} has showed that any hyperelliptic curve (in particular any curve of genus 1 or 2) satisfies the conclusion of Question~\ref{conjpowers2} for all~$m\geq 2$.

Proposition~\ref{proppowers2} below, which parallels \cite{Tavakol2} in the case of hyperelliptic curves and \cite{YinK3} in the case of K3 surfaces, shows that for a Fano or Calabi--Yau hypersurface $Y$ the only non-trivial relations among tautological cycles in powers of $Y$ are given by the MCK relation~\eqref{E:mck} and the finite-dimensionality relation. As a consequence, we obtain in Corollary~\ref{cor:TautoCubic} that any Kimura--O'Sullivan finite-dimensional cubic hypersurface (\emph{e.g.}, cubic threefolds and cubic fivefolds -- conjecturally all cubic hypersurfaces in any positive dimension since smooth projective varieties are conjecturally Kimura--O'Sullivan finite-dimensional, \emph{cf.}~\cite{KimuraFD}) satisfies the conclusion of Question~\ref{conjpowers2} for all $m\geq 2$.

	\begin{prop}\label{proppowers2} Let $Y$ be a Fano or Calabi--Yau smooth hypersurface in $\PP^{n+1}$\,; \emph{i.e.}, a hypersurface in $\PP^{n+1}$ of degree $\leq n+2$.
		 Then
		\begin{enumerate}[(i)]
\item 	Question~\ref{conjpowers2} has a positive answer for $m \leq 2$\,;
\item 	Question~\ref{conjpowers2} has a positive answer for all $m \leq 2b_{\mathrm{pr}}(Y)+1$ in case $n$ even and for all $m \leq b_{\mathrm{pr}}(Y)+1$ in case $n$ odd if and only if the Chow--K\"unneth decomposition~\eqref{eq:CKabsolute} is multiplicative\,;
\item Question~\ref{conjpowers2} has a positive answer for all $m$ if and only if the Chow--K\"unneth decomposition~\eqref{eq:CKabsolute} is multiplicative and the Chow motive of $Y$ is  Kimura--O'Sullivan finite-dimensional.
		\end{enumerate}
	\end{prop}

	Before giving the proof of the proposition, let us introduce some notations. Let $h$ be the hyperplane section class. We denote $o$
	 the class of $\frac{1}{\deg(Y)}h^{n}\in\CH_{0}(Y)$. 
	 If $Y$ is Fano, $o$ is represented by any point\,; if $Y$ is Calabi--Yau, then $o$ is the canonical 0-cycle studied in \cite{MR2916291}. 
	 For ease of notation, we write $o_i$ and $h_i$ for $p_i^*o$ and $p_i^*h$ respectively, where $p_i
	: Y^m \to Y$ is the projection on the $i$-th factor. Finally we define the following correspondence in $\CH^n(Y\times Y)$\,:
	\begin{equation}\label{E:tau}
\tau := \left\{
\begin{array}{ll}
\pi^n - \frac{1} {\deg(Y)}  h_1^{n/2}\cdot h_2^{n/2} \quad \text{if} \ n \ \text{even}\\
\pi^n \ \ \  \qquad \qquad \qquad\qquad \text{if} \ n \ \text{odd} 
\end{array}
\right.
	\end{equation}
	where $\pi^n$ is the Chow--K\"unneth projector  defined in~\eqref{eq:CKabsolute},	
	and we set $\tau_{i,j} := p_{i,j}^*\tau$, where $p_{i,j} : Y^m \to Y\times Y$ is the
	projection on the product of the $i$-th and $j$-th factors. Note that $\tau$ is
	an idempotent correspondence that commutes with $\pi^n$, and that cohomologically it is nothing but the
	orthogonal projector on the primitive cohomology $\HH^n(Y,\QQ)_{\prim}$ of $Y$.
	We write $$\h^n(Y)_{\mathrm{prim}} := \tau \h(Y) := (Y,\tau)$$ for the direct summand of $\h(Y)$ cut out by $\tau$ and call it the \emph{primitive summand} of $\h(Y)$.\medskip

	We note that Proposition~\ref{proppowers2} is trivial in the case $b_{\mathrm{pr}}(Y) =0$. Indeed in that case $Y$ is either a hyperplane (hence isomorphic to projective space) or an odd-dimensional quadric. In both cases, the Chow motive of $Y$ is known to be of Lefschetz type, so
	 that it is finite-dimensional and any Chow--K\"unneth decomposition is multiplicative (see~\cite[Theorem~2]{SV2}). From now on, we will therefore assume that $b_{\mathrm{pr}}(Y) \neq 0$.
	In order to prove Proposition~\ref{proppowers2}, we first determine as in 
	\cite[Lemma~2.3]{YinK3} the cohomological relations among the cycles introduced
	above. 
	
	\begin{lem}\label{L:Yin'}  
		The $\QQ$-subalgebra $\overline{\mathrm{R}}\,^*(Y^m)$ of the cohomology algebra ${\mathrm{H}}^*(Y^m,\QQ)$ generated by $o_i, h_i, \tau_{j,k}$, $1\leq i \leq m$, $1\leq j \neq k \leq m$, is isomorphic to the free graded $\QQ$-algebra generated by $o_i, h_i, \tau_{j,k}$, modulo the following relations\,:
			 	\begin{equation}\label{E:X'}
			h_i^n =\deg(Y)\,o_i, \quad h_i \cdot o_i = 0\,;
			\end{equation}
			\begin{equation}\label{E:X2'}
		\tau_{i,j} = 	\tau_{j,i}\,,\quad
				 \tau_{i,j} \cdot h_i = 0, \quad \tau_{i,j} \cdot \tau_{i,j} = (-1)^nb_{\mathrm{pr}}\, o_i\cdot o_j
			\,;
			\end{equation}
			\begin{equation}\label{E:X3'}
			\tau_{i,j} \cdot \tau_{i,k} = \tau_{j,k} \cdot o_i  \quad \mbox{for $\{i,j\} \neq \{i,k\}$}\,;
			\end{equation}
			\begin{equation}\label{E:X4'}
			\left\{
			\begin{array}{ll}
		\sum\limits_{\sigma \in \mathfrak{S}_{b_{\mathrm{pr}}+1}} \mathrm{sgn}(\sigma) \prod\limits_{i=1}^{b_{\mathrm{pr}}+1} \tau_{i,b_{\mathrm{pr}}+1 +\sigma(i)} = 0 \mbox{ and its permutations under $\mathfrak{S}_{2b_{\mathrm{pr}}+2}$}& \mbox{if $n$ is even}\,;\\
		
		\sum\limits_{\sigma \in \mathfrak{S}_{b_{\mathrm{pr}}+2}}  \prod\limits_{i=1}^{\frac{1}{2}b_{\mathrm{pr}}+1} \tau_{\sigma(2i-1),\sigma(2i)} = 0 & \mbox{if $n$ is odd}\,;
			\end{array}
			\right.
			\end{equation}
			where $b_{\mathrm{pr}}:=b_{\mathrm{pr}}(Y) := b^n_{\mathrm{prim}}(Y)$ is the rank of the primitive part of $\HH^n(Y, \QQ)$ (which is even if $n$ is odd).
	\end{lem}
	\begin{proof}
		The proof is directly inspired by, and closely follows, \cite[\S 3.1]{Tavakol2} and \cite[\S 3]{YinK3}. 
		
	First, we check that the above relations hold in $\HH^*(Y^m,\QQ)$.
		The relations \eqref{E:X'} take place in~$Y$ and are clear. The relations \eqref{E:X2'} take place in $Y^2$\,: the relation $	\tau_{i,j} = \tau_{j,i}$ is clear, the
		relation $ \tau_{i,j}
		\cdot h_i = 0$ follows directly from \eqref{eqn:RelationX2}, 
		while the relation
		$\tau_{i,j} \cdot \tau_{i,j} = (-1)^nb_{\mathrm{pr}}\, o_i\cdot o_j $ follows 
		from the additional general fact that $\deg(\Delta_Y\cdot \Delta_Y) = \chi (Y)$, the topological Euler
		characteristic of $Y$ which in our case is $\chi(Y) = n+1 + (-1)^nb_{\mathrm{pr}}$.
		The relation~\eqref{E:X3'} takes place in $Y^3$ and follows from the
		relations~\eqref{E:X'} and~\eqref{E:X2'} together with the fact that the cohomology algebra $\HH^*(Y^m,\QQ)$ is graded (see also Remark~\ref{R:MCK}). Finally, if $n$ is even, the relation \eqref{E:X4'} takes place in $Y^{2(b_{\mathrm{pr}}+1)}$ and expresses the fact that the $(b_{\mathrm{pr}}+1)$-th exterior power of $\HH^n(Y, \QQ)_{\mathrm{prim}}$ vanishes\,; while if $n$ is odd, 
		the relation \eqref{E:X4'} takes place in $Y^{b_{\mathrm{pr}}+2}$ and follows from the vanishing of the $(b_{\mathrm{pr}}+2)$-th symmetric power of $\HH^n(Y,\QQ)_{\mathrm{prim}}$ viewed as a super-vector space.
		
		Second, we check that these relations generate all relations among $o_i, h_i, \tau_{j,k}$, $1\leq i \leq m$, $1\leq j < k \leq m$. To that end, let $\mathcal{R}^\ast(Y^m)$ denote
		the formally defined graded $\QQ$-algebra generated by symbols $\tau_{i,j}, o_i, h_i$ (placed in appropriate degree) modulo relations \eqref{E:X'}, \eqref{E:X2'}, \eqref{E:X3'}, \eqref{E:X4'}.		
	We claim that the pairing
		 $${\mathcal{R}}\,^e(Y^m) \times {\mathcal{R}}\,^{mn-e}(Y^m) \to {\mathcal{R}}\,^{mn}(Y^m) \simeq \QQ\{\,o_1\cdot o_2 \cdots o_m\}$$ 
		 (which is formally defined by relations \eqref{E:X'}, \eqref{E:X2'}, \eqref{E:X3'}, \eqref{E:X4'}) is non-degenerate for all $0\leq e \leq mn$. Arguing as in~\cite[\S 3.1]{YinK3}
		 this claim implies that the natural homomorphism from $\mathcal{R}^\ast(Y^m)$ to $\overline{\mathrm{R}}\,^*(Y^m)$ is an isomorphism, which establishes the lemma.		 	 
To prove the claim, we observe as in \cite[\S 3.2]{YinK3} (or as in \cite[p.~2047]{Tavakol2}) that the relations \eqref{E:X'}, \eqref{E:X2'} and \eqref{E:X3'} imply that ${\mathcal{R}}^*(Y^m)$ is linearly spanned by monomials in $\{o_i\}, \{h_i^k, 0\leq k <n\}$ and $\{\tau_{i,j}\}$ with no repeated index (\emph{i.e.}~each index $i \in \{1,\ldots, m\}$ appears at most once).  Denote then (as in \cite[\S 3.4]{YinK3}) by $\Q\mathrm{Mon}_\tau^{nk}(2k) $ the free $\Q$-vector space spanned by the (formal) monomials (of degree $nk$) of the form $\tau_{i_1,i_2}\cdots \tau_{i_{2k-1},i_{2k}}$, where $\{1,\ldots,2k\} = \{i_1,i_2\}\sqcup \cdots \sqcup \{i_{2k-1},i_{2k}\}$ is a partition of $\{1,\ldots,2k\}$ into subsets of order 2.
Using a direct analogue
 of \cite[Lemma~3.3]{YinK3}, the claim reduces to showing that the kernel of the pairing  
$$\Q\mathrm{Mon}_\tau^{nk}(2k) \times \Q\mathrm{Mon}_\tau^{nk}(2k)  \longrightarrow \Q\{o_1\cdots o_{2k}\}$$
formally defined by the relations \eqref{E:X'}, \eqref{E:X2'} and \eqref{E:X3'}
 is generated by the relation~\eqref{E:X4'}. This is derived from \cite[Theorem~3.1]{HW}, and is the content of \cite[\S \S 3.5-3.7]{YinK3} in the case $n$ even and of \cite[Proposition~3.7]{Tavakol2} in the case $n$ odd. The lemma is proven.
	\end{proof}

	\begin{proof}[Proof of Proposition~\ref{proppowers2}] In view of
		Lemma~\ref{L:Yin'}, the Proposition will follow if we can establish the relations \eqref{E:X'}, \eqref{E:X2'}, \eqref{E:X3'} and \eqref{E:X4'} in $\mathrm{R}^*(Y^m)$. 
		For that purpose, let us denote as before $o= h^n/\deg(Y) \in \CH^n(Y)$, and $\tau$ as in \eqref{E:tau}.
		 The relations \eqref{E:X'} hold true  in $\mathrm{R}^*(Y)$ for any smooth hypersurface $Y$. The  relations \eqref{E:X2'} hold true in $\mathrm{R}^*(Y^2)$ for any smooth hypersurface~$Y$\,; this is a combination of \eqref{eqn:RelationX2} and of the identity $\Delta_Y\cdot \Delta_Y = \Delta_Y\cdot p_1^*c_n(Y) = \frac{\chi(Y)}{\deg(Y)} \, \Delta_Y\cdot h_1^n$.
		 The relation \eqref{E:X3'}  takes place in $\mathrm{R}^*(Y^3)$ and, given the
		 relations~\eqref{E:X'} and~\eqref{E:X2'} modulo rational equivalence, it holds if and only~if  the MCK relation~\eqref{E:mck} of
		 Proposition~\ref{prop:MCKcubic2} holds.
	In case $n$ even, the relation~\eqref{E:X4'}, which takes place in $\mathrm{R}^*(Y^{2(b_{\mathrm{pr}}+1)})$, holds if and only~if the motive of $Y$ is Kimura--O'Sullivan finite-dimensional.
	Finally, in case $n$ odd, if the motive  of $Y$ is Kimura--O'Sullivan finite-dimensional, then as a direct summand of $\h(Y)$ the motive $\h^n(Y)_{\mathrm{prim}}$ is also Kimura--O'Sullivan finite-dimensional and it follows that the $(b_{\mathrm{pr}}+1)$-th symmetric power, and hence the  $(b_{\mathrm{pr}}+2)$-th symmetric power, of~$\h^n(Y)_{\mathrm{prim}}$ vanishes. Since the cycle $\sum_{\sigma \in \mathfrak{S}_{b_{\mathrm{pr}}+2}}  \prod_{i=1}^{\frac{1}{2}b_{\mathrm{pr}}+1} \tau_{\sigma(2i-1),\sigma(2i)}$ clearly defines a cycle on $\operatorname{Sym}^{b_{\mathrm{pr}}+2}\h^n(Y)_{\mathrm{prim}}$, it vanishes. Conversely, if  $\R^{n(b_{\mathrm{pr}}+1)}(Y^{2(b_{\mathrm{pr}}+1)})$ injects in cohomology \emph{via} the cycle class map, then the tautological cycle $\sum_{\sigma \in \mathfrak{S}_{b_{\mathrm{pr}}+1}}  \prod_{i=1}^{b_{\mathrm{pr}}+1} \tau_{i,b_{\mathrm{pr}}+1 +\sigma(i)}$, which defines an idempotent correspondence on $(\h^n(Y)_{\mathrm{prim}})^{\otimes (b_{\mathrm{pr}}+1)}$ with image $\operatorname{Sym}^{b_{\mathrm{pr}}+1}\h^n(Y)_{\mathrm{prim}}$, is rationally trivial since $\operatorname{Sym}^{b_{\mathrm{pr}}+1}\HH^n(Y,\QQ)_{\mathrm{prim}}=0$.
	\end{proof}

Finally, by combining Proposition~\ref{proppowers2}$(ii)$ with Theorem~\ref{thm:MCKcubic},  we have the following result in the case of cubic hypersurfaces\,:
\begin{cor}\label{cor:TautoCubic}
Let $Y$ be a smooth cubic hypersurface in $\PP^{n+1}$. Then the tautological ring
${\mathrm{R}}^\ast(Y^m):=\langle p_i^\ast h, \ 
		p_{k,l}^\ast
		\Delta_Y
		\rangle\ \subset\ \CH^\ast(Y^m)$
		injects into $\HH^{2*}(Y^m,\QQ)$ \emph{via} the cycle class
		map for all $m \leq 2b_{\mathrm{pr}}(Y)+1$ if $n$ is even and  for all $m \leq b_{\mathrm{pr}}(Y)+1$ if $n$ is odd.
		 Moreover, ${\mathrm{R}}^\ast(Y^m)$ injects into $\HH^{2*}(Y^m,\QQ)$ \emph{via} the cycle class
		map for all $m$ if and only if $Y$ is Kimura--O'Sullivan finite dimensional.
		\qed
\end{cor}

\begin{rmk}\label{rmk:Cubic3and5folds}
	 We note that all smooth projective curves, and in particular cubic curves, are Kimura--O'Sullivan finite-dimensional, and that cubic surfaces have trivial Chow groups (Definition~\ref{def:trivialChow}) and hence are Kimura--O'Sullivan finite-dimensional by Lemma~\ref{trivial} and by the fact that Lefschetz motives are finite-dimensional.
	 In addition, cubic threefolds and fivefolds are known to be Kimura--O'Sullivan finite dimensional\,; see \cite{MR2877438} and \cite[Example~4.12]{MR3141813}, respectively. Hence, for $n=1, 2, 3,5$, Corollary \ref{cor:TautoCubic} gives the injectivity
${\mathrm{R}}^\ast(Y^m)\hookrightarrow \HH^{2*}(Y^m,\QQ)$ for all~$m$. We note further that in case $n=1$, after fixing a closed point $O$ in $Y$, the pair $(Y,O)$ defines an elliptic curve; in this case the injectivity
${\mathrm{R}}^\ast(Y^m)\hookrightarrow \HH^{2*}(Y^m,\QQ)$ for all~$m$ is due to Tavakol~\cite{Tavakol3}, and can also be deduced directly from O'Sullivan's theory of symmetrically distinguished cycles~\cite{OSul}.
\end{rmk}

	\subsection{On the extended tautological ring of smooth  cubic fourfolds} This paragraph is not needed in the rest of the paper\,; its aim is to show how the arguments of \S \ref{SS:cohotauto} can be refined to establish analogues of \cite[Theorem]{YinK3} concerning K3 surfaces in the case of cubic fourfolds.
	
	For a K3 surface $S$, let $\tilde{\mathrm{R}}^\ast(S^m)\subset \CH^\ast(S^m)$ be
	the
	$\QQ$-subalgebra generated by $\CH^1(S)$ and the diagonal $\Delta_S\in
	\CH^2(S\times
	S)$. Voisin conjectures \cite[Conjecture 1.6]{V17} that
	$\tilde{\mathrm{R}}^\ast(S^m)$ injects
	into cohomology (by~\cite{YinK3} this is equivalent to Kimura--O'Sullivan finite-dimensionality
	of $S$). Here is a version of Voisin's conjecture for cubic fourfolds that refines Question~\ref{conjpowers2}\,:
	
	\begin{conj}\label{conjpowers} Let $Y$ be a smooth cubic fourfold, and
		$m\in\N$. Let $\tilde{\mathrm{R}}^\ast(Y^m)$ be the $\QQ$-subalgebra
		\[ \tilde{\mathrm{R}}^\ast(Y^m):=\langle p_i^\ast h, \ p_j^\ast \CH^2(Y), \
		p_{k,l}^\ast
		\Delta_Y
		\rangle\ \ \ \subset\ \CH^\ast(Y^m),\ \]
		where $p_i$, $p_j$ and $p_{k,l}$ denote the various projections from $Y^m$ to
		$Y$
		and to $Y^2$.
		Then $\tilde{\mathrm{R}}^*(Y^m)$ injects into $\HH^{2*}(Y^m,\QQ)$ \emph{via} the cycle class
		map.
	\end{conj}
	
	In what follows, $b_{\mathrm{tr}}(Y)$ denotes the dimension of the transcendental cohomology of the smooth cubic fourfold $Y$, \emph{i.e.}, the dimension of the orthogonal complement in $\HH^*(Y, \QQ)$ of the subspace consisting of Hodge classes.
	Using the multiplicative Chow--K\"unneth relation \eqref{E:mck} for cubic hypersurfaces,  we can adapt the method of Yin concerning K3 surfaces~\cite{YinK3} and prove the following result.
	
	\begin{prop}\label{proppowers} Let $Y$ be a smooth cubic fourfold. Then
		Conjecture \ref{conjpowers} is true  for $Y$ for all $m \leq 2b_{\mathrm{tr}}(Y)+1$, in particular for all $m\leq 5$.
Moreover, Conjecture~\ref{conjpowers} is true for $Y$ for all $m$ if and only~if $Y$ is Kimura--O'Sullivan finite-dimensional.
	\end{prop}

	Before giving the proof of Proposition~\ref{proppowers}, let us introduce some notations. We fix a smooth cubic fourfold $Y$. First we note that the cycle class map
	$\CH^2(Y) \to \HH^4(Y, \QQ)$ is injective and an isomorphism on the Hodge classes in
	$\HH^4(Y, \QQ)$. Let $\{l^s\}_s$ be an orthogonal basis of the Hodge classes in
	$\HH^4(Y, \QQ)_{\mathrm{prim}}$, \emph{i.e.}, $\{h^2\} \cup \{l^s\}_s$ forms a basis of
	$\CH^2(Y)$ with the property that $l^s\cdot l^{s'} = 0$ whenever $s\neq s'$ and
	$l^s\cdot h = 0$ for all $s$. Note that the latter property $l^s\cdot h = 0$
	holds cohomologically (by definition of primitive cohomology) and lifts to
	rational equivalence since $l^s\cdot h$ must be a rational multiple of $h^3$ in
	$\CH^3(Y)$. Recall that all points on $Y$ are rationally equivalent and that $o$
	denotes the class of any point on $Y$. For ease of notation, we write $o_i,
	l_i^s$ and $h_i$ for $p_i^*o$, $p_i^*l^s$ and $p_i^*h$ respectively, where $p_i
	: Y^m \to Y$ is the projection on the $i$-th factor. Finally we define 
	$$\tau := \pi^4 - \frac{1}{3}h_1^2\cdot h_2^2 - \sum_s \frac{l^s_1\cdot
		l^s_2}{\deg(l^s\cdot l^s)} \quad \in \CH^4(Y\times Y),$$
	where $\pi^4$ is the  Chow--K\"unneth projector $\pi^4_{Y}$ defined in~\eqref{eq:CKabsolute} (in our case, $n=4$),	
	and we set $\tau_{i,j} = p_{i,j}^*\tau$, where $p_{i,j} : Y^m \to Y\times Y$ is the
	projection on the product of the $i$-th and $j$-th factors. Note that $\tau$ is
	an idempotent correspondence, and that cohomologically it is nothing but the
	orthogonal projector on the transcendental cohomology of $Y$, \emph{i.e.}, on the
	orthogonal complement of the space of Hodge classes in $\HH^*(Y, \QQ)$.\medskip

	In order to prove Proposition~\ref{proppowers}, we establish as in 
	\cite[Lemma~2.3]{YinK3} sufficiently many relations among the cycles introduced
	above. Central is 	the MCK relation~\eqref{E:mck} of Proposition~\ref{prop:MCKcubic2}.
	\begin{lem}\label{L:Yin}
		In $\widetilde{\mathrm{R}}^*(Y^m)$ we have relations
		\begin{equation}\label{E:X}
		h_i \cdot o_i = 0,  \quad
		  l_i^s\cdot h_i = 0, \quad
		h_i^4 = 3o_i, \quad l_i^s\cdot l_i^s  = \deg(l^s\cdot l^s)\,o_i\,;
		\end{equation}
		\begin{equation}\label{E:X2}
			\tau_{i,j} = 	\tau_{j,i}\,,\quad
				 \tau_{i,j} \cdot h_i = 0, \quad \tau_{i,j} \cdot
		l_i^s = 0, \quad \tau_{i,j} \cdot \tau_{i,j} = b_{{\mathrm{tr}}}\, o_i\cdot o_j
		\,;
		\end{equation}
		\begin{equation}\label{E:X3}
		\tau_{i,j} \cdot \tau_{i,k} = \tau_{j,k} \cdot o_i  \quad \mbox{for $\{i,j\} \neq \{i,k\}$,} 
		\end{equation}
		where $b_{{\mathrm{tr}}}:=b_{\mathrm{tr}}(Y)$ is the rank of the transcendental part of $\HH^4(Y, \QQ)$.
	\end{lem}
	\begin{proof}
		The relations \eqref{E:X} take place in $Y$ and were established in the
		discussion above the lemma. The relations \eqref{E:X2} take place in $Y^2$.
		 The relation $\tau_{i,j} = 	\tau_{j,i}$ is trivial and the relation $ \tau_{i,j}
		\cdot h_i = 0$ follows directly from~\eqref{eqn:RelationX2}. The relation
		$\tau_{i,j} \cdot \tau_{i,j} = b_{{\mathrm{tr}}}\, o_i\cdot o_j $ follows 
		from the general fact that $\deg(\Delta_Y\cdot \Delta_Y) = \chi (Y)$, the topological Euler
		characteristic of $Y$, and the fact that $\CH^8(Y\times Y)$ is 1-dimensional.
		Concerning the relation $\tau_{i,j} \cdot l_i^s =0$, this is a consequence of  
		$\Delta_Y \cdot p_1^*l^s = (\delta_Y)_* l^s = p_1^*l^s\cdot p_2^*o + p_1^*o\cdot
		p_2^*l^s$, where $\delta_Y$ is the small diagonal of $Y^{3}$, seen as a correspondence from $Y$ to $Y\times Y$
		and  where we used the relation $h\cdot l^s = 0$ together with
		the MCK relation~\eqref{E:mck} of Proposition~\ref{prop:MCKcubic2}.
		Finally, the relation \eqref{E:X3} takes place in $Y^3$ and follows from the
		relations~\eqref{E:X} and~\eqref{E:X2} together with the MCK relation~\eqref{E:mck} of
		Proposition~\ref{prop:MCKcubic2}.
	\end{proof}

	\begin{proof}[Proof of Proposition~\ref{proppowers}] 
		Once we have the key
		Lemma~\ref{L:Yin}, the proof of the proposition 
is similar to that of Proposition~\ref{proppowers2}\,:
		 beyond the cohomological relations \eqref{E:X} and \eqref{E:X2}, the only other relations in the $\QQ$-algebra $\HH^*(Y^m, \QQ)$ are the ones given by \eqref{E:X3} (expressing that $\HH^*(Y, \QQ)$ is a graded $\Q$-algebra) and by the finite-dimensionality relation $\Lambda^{b_{\mathrm{tr}}+1}\HH^*(Y, \QQ)_{\mathrm{tr}} = 0$ (expressing that the transcendental cohomology $\HH^*(Y, \QQ)_{\mathrm{tr}}$ has finite dimension $b_{\mathrm{tr}}$).
	\end{proof}

	\subsection{The Franchetta property for powers of cubic hypersurfaces\,: Proof of Theorem~\ref{thm2:Fcubic}} \label{SS:cubic}
	
	Let $Y$ be a smooth cubic hypersurface.
	We start by using Proposition~\ref{strat} to determine generators of
	$\operatorname{GDCH}_B^*(Y^m)$ for $m\leq 4$\,:

	\begin{prop}\label{P:star4} Let $B\subset\PP\HH^0(\PP^{n+1},\mathcal{O}(3))$ be the open subset
		parameterizing smooth cubic hypersurfaces of dimension~$n$, and let $\YY \to B$ be the universal
		family. For all $m\leq 4$, we have
		$$\operatorname{GDCH}_B^*(Y^m) = \langle p_i^*h, p_{j,k}^*\Delta_Y\rangle,$$
		where $p_i : Y^m \to Y$ and $p_{j,k}: Y^m \to Y\times Y$ are the various projections. Here, $ \operatorname{GDCH}^\ast_B(Y^m)$ is as in~\eqref{E:Fm}.
	\end{prop}
	
	\begin{proof} In view of Proposition~\ref{strat},
		we simply check that $(\PP^{n+1},\OO(3))$ satisfies the condition $(\star_4)$. 
		Let $\Delta_{i,j}:= p_{i,j}^{-1}(\Delta_{\PP^{n+1}})$ be a big diagonal in $(\mathbb{P}^{n+1})^4$.
		Since all the closed orbits of the natural action of $\operatorname{PGL}_{n+2}$ on $(\PP^{n+1})^{4}\backslash (\bigcup_{i,j} \Delta_{i,j})$ parameterize four collinear points, we only need to check $(\star_4)$ for four collinear points 
		$x_1,\ldots, x_4$. 
		In this case, the needed surjectivity follows from surjectivity of the restriction
		and the evaluation
		$$\HH^0(\PP^{n+1},\mathcal O(3)) \twoheadrightarrow \HH^0(\PP^1 ,\mathcal O(3))
		\twoheadrightarrow \bigoplus_{i=1}^4 \C_{x_i}\ ,$$
		where $\PP^1$
		is the line containing these points. 
	\end{proof}

	\begin{proof}[Proof of Theorem~\ref{thm2:Fcubic}]
		By Proposition \ref{P:star4},  when $m\leq 4$, $\operatorname{GDCH}_B^*(Y^m)$ is generated by tautological cycles, \emph{i.e.}~$\operatorname{GDCH}_B^*(Y^m) = \mathrm{R}^*(Y^m)$.  We then note that Corollary~\ref{cor:TautoCubic} and Remark~\ref{rmk:Cubic3and5folds} give the injectivity of ${\mathrm{R}}^\ast(Y^m)\to \HH^{2*}(Y^m,\QQ)$	for all $m$ if $n=1, 2, 3, 5$, for $m\leq 45$ if $n$ is even and for $m\leq 171$ if $n$ is odd (indeed, for $n>2$ even we have $b_{\mathrm{pr}}(Y)\ge 22$ and for $n>5$ odd we have $b_{\mathrm{pr}}(Y)\ge 170$, see \emph{e.g.}~\cite[Corollary 1.11]{Huy}).
	\end{proof}

\subsection{The Franchetta property and the cancellation property for Chow motives}
Due to Theorem~\ref{thm2:Fcubic}, the following motivic proposition applies to cubic hypersurfaces.

\begin{prop}\label{P:cancellation}
Let $\mathcal X \to B$ be a smooth projective family parameterized by a smooth quasi-projective variety $B$. Let $X :=\mathcal X_b$ be a closed fiber and consider the additive thick subcategory $\MM_X$ of $\MM_{\mathrm{rat}}$ generated by the Tate twists $\h(X)(n)$, $n\in \mathbb Z$, with morphisms given by generically defined correspondences.
 Assume that $X$ satisfies the standard conjectures and that $\mathcal X^2_{/B} \to B$ has the Franchetta property.
Then $\MM_X$ is semi-simple. In particular, cancellation holds, \emph{i.e.}, if we have $A\oplus A_{1} \simeq A\oplus A_{2}$ in $\MM_X$, then $A_{1}\simeq A_{2}$ in $\MM_X$.
\end{prop}
\begin{proof}
By definition, the objects of $\MM_X$ are of the form $(Z,p,n)$ 
with $Z = \sqcup_i X\times \PP^{n_i}$ for finitely many $n_i \in \mathbb Z_{\geq 0}$, $p\in \End_{\MM_{\mathrm{rat}}}(\h(Z))$ an idempotent and generically defined correspondence, 
and $n$ an integer\,; and the morphisms $$\operatorname{Hom}_{\MM_X}((Z_1,p_1,n_1), (Z_2,p_2,n_2))\subseteq \operatorname{Hom}_{\MM_{\mathrm{rat}}}((Z_1,p_1,n_1), (Z_2,p_2,n_2))$$  are given by the subspace 
consisting of generically defined correspondences.

By the Franchetta property for $\mathcal X^2_{/B} \to B$ and the coincidence of homological and numerical equivalence on $X\times X$, the restriction of the functor $\MM_{\mathrm{rat}} \to \MM_{\mathrm{num}}$ to $\MM_X$ is fully faithful. Let us denote $\overline{\MM}_X$ its essential image. We have to show that $\overline{\MM}_X$ is semi-simple. 
This follows simply from the fact that, for $M\in \overline{\MM}_X$, $\End_{\overline{\MM}_X}(M) $ is a sub-algebra of the algebra $\End_{\MM_{\mathrm{num}}}(M)$ which is semi-simple by the main theorem of \cite{JannsenNumerical}.
\end{proof}

\subsection{Application to the motive of the Fano variety of lines on a smooth cubic hypersurface}
We start with the observation that the Chow--K\"unneth decomposition~\eqref{eq:CKabsolute} for smooth hypersurfaces is generically defined in the following sense.
Let
$$B\subset\PP\HH^0(\PP^{n+1},\mathcal{O}(d))$$ be the open subset
parameterizing smooth hypersurfaces of degree $d$ in $\PP^{n+1}$, and let $\YY \to B$ be the universal
family.  If $H \in \CH^1(\YY)$ denotes the relative
hyperplane section, then the relative correspondences
\begin{equation}\label{eq:CKcubic}
\pi^{2i}_\YY := \frac{1}{d} H^{n-i} \times_B H^i, \quad 2i\neq n, 	\quad	
\text{and} \quad \pi^n_\YY := \Delta_{\YY/B} - 	\sum_{2i\neq n} \pi^{2i}_\YY
\end{equation}
in $\CH^n(\YY \times_B \YY)$ define a relative Chow--K\"unneth decomposition, in the sense 
that they are relative idempotent correspondences
 and their
specializations to any fiber $\YY_b$ over $b \in B$ gives a Chow--K\"unneth
decomposition of $\YY_b$, and in fact restrict to the Chow--K\"unneth decomposition~\eqref{eq:CKabsolute} for all $b\in B$.
Moreover, the idempotent correspondence $\tau$ of~\eqref{E:tau} defining the direct summand $\h^n(Y)_{\prim}$ is also generically defined with respect to the family $\YY \to B$\,; indeed, the relative correspondence 
	\begin{equation}\label{E:tau'}
\tau_\YY := \left\{
\begin{array}{ll}
\pi_\YY^n - \frac{1} {d}  H_1^{n/2}\cdot H_2^{n/2} \quad \text{if} \ n \ \text{even}\\
\pi_\YY^n \ \ \  \qquad \qquad \qquad\qquad \text{if} \ n \ \text{odd} 
\end{array}
\right.
\end{equation}
defines a relative idempotent correspondence whose specialization to any fiber $\YY_b$ over $b\in B$ is the idempotent correspondence $\tau$ of~\eqref{E:tau}.
\medskip

	Let us now focus on the case where $Y$ is a smooth cubic hypersurface in $\PP^{n+1}$. Let $F$ be its Fano variety of lines, which is known to be smooth projective of dimension $2n-4$. As before, we denote  $B\subset\PP\HH^0(\PP^{n+1},\mathcal{O}(3))$  the Zariski open subset
	parameterizing smooth cubic hypersurfaces of dimension~$n$\,; and we denote $\YY \to B$ and $\mathcal F \to B$ the corresponding universal families.
	The first isomorphism in the following theorem is a motivic lifting of \cite{GalkinShinder}. It refines, and gives a new proof of, the main result of~\cite{Lat-Y-FY}.
\begin{thm}\label{thm:FY2}
Notation is as above. 
\begin{enumerate}[$(i)$]
\item We have an isomorphism of Chow motives
\begin{equation}\label{gs}
\h(F)\simeq \Sym^{2}(\h^{n}(Y)_{\prim}(1))\oplus \bigoplus_{i=1}^{n-1}\h^{n}(Y)_{\prim}(2-i)\oplus \bigoplus_{k=0}^{2n-4}\1(-k)^{\oplus a_{k}}\ ,
\end{equation} 
where 
\[
a_{k}=\begin{cases}
\lfloor \frac{k+2}{2}\rfloor &\text{ if } k<n-2 \\
\lfloor \frac{n-2}{2}\rfloor &\text{ if } k=n-2\\
\lfloor \frac{2n-2-k}{2}\rfloor &\text{ if } k>n-2.\\
\end{cases}
\]
\item Denoting $N$ the Chow motive $(Y,\Delta_Y - o\times Y - Y\times o)$ for any choice of point $o\in Y$, we have an isomorphism of Chow motives
$$\qquad \h(F)\oplus \mathds{1}(2-n) \simeq   \mathrm{Sym}^2(N(1)).$$ 
\item We have an isomorphism of Chow motives $$\h(F)(-2) \oplus \h(Y) \oplus \h(Y)(-n) \simeq \mathrm{Sym}^2\h(Y).$$
\item $F$ and $Y$ have canonical Chow--K\"unneth decompositions, and 
\begin{equation}\label{hn-2}
\h^{n-2}(F)\simeq 
\begin{cases}\h^{n}(Y)_{\prim}(1)\oplus \1(-\frac{n-2}{2})^{\oplus \lfloor\frac{n+2}{4}\rfloor}, &\text{ if } $n$ \text{ is even;}\\
\h^{n}(Y)_{\prim}(1)= \h^{n}(Y)(1)&\text{ if } $n$ \text{ is odd},
\end{cases}
\end{equation} 
where the isomorphism is given by $P^{*}: \h^{n}(Y)_{\prim}(1)\to \h^{n-2}(F)$ and in the even case, for the $i$-th copy, $\cdot g^{\frac{n}{2}+1-2i}c^{i-1}: \1(-\frac{n-2}{2}) \to \h^{n-2}(F)$,
where $1\leq i\leq \lfloor\frac{n+2}{4}\rfloor$, $P = \{(y,\ell) \in Y\times F \ \vert\ y\in \ell\}$ is the incidence correspondence and $g:=-c_{1}(\mathcal{E}|_{F})$, $c:=c_{2}(\mathcal{E}|_{F})$ with $\mathcal{E}$ being the rank-2 tautological bundle on the Grassmannian $\Gr(\PP^{1}, \PP^{n+1})$. 
\item If $n=4$, cup-product induces an isomorphism of Chow motives
$$\mathrm{Sym}^2\h^{2}(F)  \stackrel{\simeq}{\longrightarrow} \h^{4}(F).$$
\end{enumerate}
\end{thm}
\begin{proof} 

Our starting point is the isomorphism of Chow motives
\begin{equation}\label{E:isoGSV}
\bigoplus_{i=0}^n \h(Y)(-i) \oplus \h(F)(-3) \oplus \h(F)(-2)^{\oplus 2} \oplus \h(F)(-1)
  \simeq
\h(Y^{[2]}) \oplus  \h(F)(-3) \oplus \h(F)(-2) \oplus \h(F)(-1),
\end{equation}
which can be obtained by applying the blow-up formula and the projective bundle formula for Chow motives to the construction due to Galkin--Shinder~\cite{GalkinShinder} and Voisin~\cite{VoisinJEMS}\,; see \cite[Diagram (15) and Equation~(16)]{FLV-MCK}.

We wish to apply Proposition \ref{P:cancellation} to
\[ X := Y\times Y \sqcup F ,\] 
seen as a fiber of the family $(\YY\times_B \YY) \sqcup \mathcal F \to B$.
Since the Galkin--Shinder--Voisin construction can be done in a relative setting, all arrows in the blow-up formula and the projective bundle formula are generically defined (with respect to $B$), and the above isomorphism~\eqref{E:isoGSV} therefore takes place in $\MM_{X}$ (as defined in the statement of Proposition~\ref{P:cancellation}).
Let us now check that all conditions of Proposition \ref{P:cancellation} are verified for this $X$.

\noindent First, the standard conjectures hold for $X$. We only need to check them for $Y$ and for $F$\,:
	\begin{itemize}
		\item For $Y$, this is elementary.
		\item For $F$, this is either \cite[Corollary 6]{LatRM}
		or alternatively \cite[Corollary~4.3]{FLV-MCK} for a new self-contained proof.
	\end{itemize}
	\item Second, the Franchetta property holds for $X$. It is enough to show it for $\YY^4_{/B}$,  $\mathcal F\times_B \mathcal F$, and  $\mathcal F\times_B \YY\times_{B} \YY$\,:
	\begin{itemize}
		\item For $\YY^4_{/B}$, this is Theorem \ref{thm2:Fcubic}.
		\item For $\mathcal F\times_B \mathcal F$, this was achieved in our previous work \cite[Theorem~4.2]{FLV-MCK}.
		\item For $\mathcal F\times_B \YY\times_{B} \YY$. As in the proof of Proposition~\ref{P:cancellation}, denote $\overline{\MM}_X$ the semi-simple category that is the essential image of ${\MM}_X$ in ${\MM}_{\mathrm{num}}$. Looking at the reduction modulo numerical equivalence of the isomorphism~\eqref{E:isoGSV} (which takes place in $\MM_X$), and using the semisimplicity of $\MM_{\mathrm{num}}$ \cite{JannsenNumerical}, we obtain a split injective morphism
		\[  \h(F) \to \h(Y^{[2]})(2)\ \ \ \hbox{in}\ \overline{\MM}_X.\]
		Using the standard conjectures for $F$ and $Y$, combined with the Franchetta property for $F\times F$ \cite[Theorem~4.2]{FLV-MCK}, this lifts to a split injection of Chow motives
		\[  \h(F)\to \h(Y^{[2]})(2)\ \ \ \hbox{in}\ \MM_{X} .\]
		Hence $\h(F\times Y^2)$ is a direct summand of  $\h(Y^{[2]}\times Y^2)(2)$ via a generically defined correspondence. It follows 
		that the Franchetta property for $\FF\times_B \YY\times_B \YY$ is implied by that for $\YY^m_{/B}$ with $m\le 4$, which is Theorem \ref{thm2:Fcubic}. 
	\end{itemize}
	With all conditions of 
	Proposition~\ref{P:cancellation} verified for $X$, we deduce that the category $\MM_X$ is semi-simple and in particular, the cancellation property holds.
	We obtain \eqref{gs} by cancelling an isomorphic direct summand from both sides of \eqref{E:isoGSV}.
		
	Statement $(\rom2)$ follows from $(\rom1)$, by writing
	\[ N= \h^n(Y)_{\prim}\oplus \bigoplus_{j=1}^{n-1} \1(-j).\]
	
	Likewise, statement $(\rom3)$ follows from $(\rom1)$ by writing
	\[ \h(Y)= \h^n(Y)_{\prim}\oplus \bigoplus_{j=0}^{n}\1(-j) .\]		
	
	For statement $(\rom4)$, 
	we observe that $F$ has a generically defined K\"unneth decomposition\,; using the Franchetta property for $F\times F$, this is a generically defined Chow--K\"unneth decomposition.
The isomorphism \eqref{hn-2} is generically defined and holds true in cohomology
\cite[Proposition~4.8]{FLV-MCK}. As such, the isomorphism \eqref{hn-2} holds in $\overline{\MM}_X\subset \MM_{\mathrm{num}}$. But $\MM_X\to \overline{\MM}_X $ is fully faithful, proving~$(\rom4)$.	

Statement $(\rom5)$ is proven similarly, using as input the well-known fact that cup-product induces an isomorphism $\HH^4(F,\QQ)\cong\mathrm{Sym}^2 \HH^{2}(F,\QQ)$, and the Franchetta property for $F\times F$.
\end{proof}

	\section{The generalized Franchetta conjecture for LLSS eightfolds}\label{S:proof}
	
	Given a smooth cubic fourfold $Y\subset \PP^{5}$ not containing a plane, Lehn--Lehn--Sorger--van Straten \cite{LLSvS} constructed a hyper-K\"ahler eightfold $Z=Z(Y)$ using the twisted cubic curves in $Y$. 
	Let $B$ be the open subset of $\PP\HH^0(\PP^5,\mathcal{O}(3))$ parameterizing smooth cubic fourfolds and let $B^{\circ}$ be the open subset of $B$ parameterizing those not containing a plane. Let $\YY\to B$ be the universal family of cubic fourfolds and $\mathcal{Z}\to B^{\circ}$ be the universal family of LLSS eightfolds.  The following theorem establishes Theorem~\ref{thm2:FranchettaLLSS}.
	
		\begin{thm}\label{thm:FranchettaLLSS}
		Let $\mathcal{C}^{\circ}$ be the moduli stack of smooth cubic fourfolds not containing a plane. Then the universal family of LLSS hyper-K\"ahler eightfolds \cite{LLSvS}, denoted by $\mathcal{Z}\to \mathcal{C}^{\circ}$, satisfies the Franchetta property. 
	\end{thm}
	\begin{proof}
	By \cite[Remark 2.6]{FLV}, it is sufficient to prove the Franchetta property for the universal family $\mathcal{Z}\to B^{\circ}$. For any cubic fourfold $Y$, let $\mathcal{A}_{Y}:=\langle \mathcal{O}_{Y}, \mathcal{O}_{Y}(1), \mathcal{O}_{Y}(2) \rangle^{\perp}$ be the Kuznetsov component of $\D^{b}(Y)$, which is a K3 category. 
	A natural stability condition $\sigma\in \Stab^{\dagger}(\mathcal{A}_{Y})$ is constructed in \cite{BLMS}. The Mukai lattice of $\mathcal{A}_{Y}$ always contains the $A_{2}$-lattice generated by $\lambda_{1}$ and $\lambda_{2}$, where $\lambda_{i}$ is the (cohomological) Mukai vector of $\delta(\mathcal{O}_{L}(i))$ with $\delta\colon \D^{b}(Y)\to \mathcal{A}_{Y}$ the projection functor and $L$  a (any) line in $Y$.
	
	For any cubic fourfold $Y$ not containing a plane, the result of Li--Pertusi--Zhao \cite[Theorem 1.2]{LPZ18} says that the LLSS hyper-K\"ahler eightfold $Z(Y)$ associated to $Y$ is isomorphic to $\mathcal{M}_{\sigma}(\mathcal{A}_{Y}, 2\lambda_{1}+\lambda_{2})$, the moduli space of $\sigma$-stable objects in $\mathcal{A}_{Y}$ with Mukai vector $2\lambda_{1}+\lambda_{2}$ (alternatively, one can reduce the proof of Theorem \ref{thm:FranchettaLLSS} to the very general cubic fourfold $Y$, for which the modular construction of $Z(Y)$ was already done in \cite[Main Theorem]{LLMS}.)
Theorem \ref{thm:NCK3} applied in this special case implies that there exists a split injection of Chow motives\,:
	\begin{equation}\label{eqn:SplitInj}
	\h(Z(Y))\simeq \h(\mathcal{M}_{\sigma}(\mathcal{A}_{Y}, 2\lambda_{1}+\lambda_{2}))\hookrightarrow \bigoplus_{i=1}^{r}\h(Y^{4})(l_{i})
	\end{equation}
	The theory of stability conditions in family has recently been worked out in \cite{BLMNPS}, and as such the isomorphism of Li--Pertusi--Zhao can be formulated in a relative setting. More precisely, the family $\mathcal{Z}\to B^{\circ}$ is isomorphic, as $B^{\circ}$-scheme, to the relative (smooth and projective) moduli space $\mathcal{M}\to B^{\circ}$, whose fiber over $b$ is $\mathcal{M}_{\sigma}(\mathcal{A}_{Y_{b}}, 2\lambda_{1}+\lambda_{2})$. It is also clear from the proof of Theorem \ref{thm:NCK3} that the injection \eqref{eqn:SplitInj}, as well as its left inverse, is generically defined over $B^{\circ}$ and gives rise to the following morphism of 
	 relative Chow motives over $B^{\circ}$ which is fiberwise a split injection\,:
	\begin{equation}\label{eqn:SplitInjFamily}
	\h(\mathcal{Z})\simeq \h(\mathcal{M}) \rightarrow \bigoplus_{i=1}^{r} \h(\mathcal{Y}^{4/{B^{\circ}}})(l_{i}).
	\end{equation}
 	 As a consequence, for any $b\in B^{\circ}$, we have the following commutative diagram
	 \begin{equation}\label{diag:compare}
	 \xymatrix{
	  \operatorname{GDCH}^{*}(Z_{b}) \ar[r]^{\operatorname{cl}} \ar[d]& \HH^{*}(Z_{b}, \Q)\ar[d]\\
	  \bigoplus_{i=1}^{r}\operatorname{GDCH}^{*}(Y_{b}^{4}) \ar[r]^{\operatorname{cl}}& \bigoplus_{i=1}^{r}\HH^{*}(Y_{b}^{4}, \Q),
	 }
	 \end{equation}
	 where $\operatorname{GDCH}^{*}(Z_{b}):=\ima(\CH^{*}(\mathcal{Z})\to \CH^{*}(Z_{b}))$ is the group of generically defined cycles. 
	 In the above diagram \eqref{diag:compare}, the two vertical arrows are injective by \eqref{eqn:SplitInjFamily}, the bottom arrow is injective by Theorem~\ref{thm2:Fcubic}. Therefore the top arrow is also injective, which is the content of the Franchetta property for the family $\mathcal{Z}\to B^{\circ}$.
	\end{proof}

	\section{Further results}\label{S:moduli}
	
	The aim of this section is to show how the results of Section~\ref{S:motive-moduli} also make it possible to establish the Franchetta property and to deduce Beauville--Voisin type results for certain non locally complete families of hyper-K\"ahler varieties. These include certain moduli spaces of sheaves on K3 surfaces (Corollary~\ref{cor:n(g)}), 
	and certain O'Grady tenfolds (Theorem~\ref{thm:OG10}).
	We also include a Beauville--Voisin type result for Ouchi eightfolds (Corollary~\ref{C:ouchi}).
	
	\subsection{The Franchetta property for some moduli spaces of sheaves on K3 surfaces}
	\label{SS:moduliK3}
	
We show how Theorem~\ref{thm:NCK3}  makes it possible to extend our previous results~\cite[Theorems~1.4 and~1.5]{FLV} on the Franchetta property for certain Hilbert schemes of points on K3 surfaces of small genus to the case of certain moduli spaces of sheaves on K3 surfaces of small genus. 
	
	Let $\mathcal{F}_{g}$ be the moduli stack of polarized K3 surfaces of genus $g$ and let $\mathcal{S}\to \mathcal{F}_g$ be the universal family. Denote by $H$ the universal ample line bundle of fiberwise self-intersection number $2g-2$.
	
	\begin{thm}\label{thm:MvS}
		Let $m$ be a positive integer. If $\mathcal{S}^{m}_{/\mathcal{F}_g}\to \mathcal{F}_g$ satisfies the Franchetta property, then for any $r, d, s\in \N$  such that $0\leq d^{2}(g-1)-rs+1\leq  m$ and $\gcd(r, d, s)=1$, the relative moduli space $\mathcal{M}\to \mathcal{F}_g$ of $H$-stable sheaves with primitive Mukai vector $v =(r, dH, s)$ also satisfies the Franchetta property.
	\end{thm}
	\begin{proof}
		For a given K3 surface $S$ of genus $g$, the moduli space $M_{H}(S, v)$ is a projective hyper-K\"ahler variety of dimension $v^{2}+2=d^{2}(2g-2)-2rs+2\leq 2m$. By Theorem \ref{thm:NCK3}, we have the following split injective morphism of Chow motives
		$$\h(\mathcal{M}_{H}(S, v))\hookrightarrow \bigoplus_{i=1}^{r}\h(S^{m})(l_{i}).$$
		It is clear from the proof of Theorem \ref{thm:NCK3} that the above split injective morphism, as well as its left inverse, is generically defined over $\mathcal{F}_g$. We have therefore a morphism of relative Chow motives (over $\mathcal{F}_g$) that is fiberwise a split injection\,:
		\begin{equation}\label{eqn:SplitInjFamily1}
		\h(\mathcal{M})\hookrightarrow \bigoplus_{i=1}^{r}\h(\mathcal{S}^{m}_{/\mathcal{F}_g})(l_{i}).
		\end{equation}
		As a consequence, for any $b\in \mathcal{F}_g$, we have the following commutative diagram
		\begin{equation}\label{diag:compare1}
		\xymatrix{
			\operatorname{GDCH}^{*}(M_{H}(S_{b}, v)) \ar[r]^{\operatorname{cl}}\ar[d]& \HH^{*}(M_{H}(S_{b}, v), \Q)\ar[d]\\
			\bigoplus_{i=1}^{r}\operatorname{GDCH}^{*}(S_{b}^{m}) \ar[r]^{\operatorname{cl}}& \bigoplus_{i=1}^{r}\HH^{*}(S_{b}^{m}, \Q),
		}
		\end{equation}
		where $\operatorname{GDCH}^{*}(M_{H}(S_{b}, v)):=\ima\bigl(\CH^{*}(\mathcal{M})\to \CH^{*}(M_{H}(S_{b}, v))\bigr)$ is the group of generically defined cycles relative to $\FF_g$.
		In the above diagram \eqref{diag:compare1}, the two vertical arrows are injective by \eqref{eqn:SplitInjFamily1}, the bottom arrow is injective by hypothesis. Therefore the top arrow is also injective, which is the content of the Franchetta property for the family $\mathcal{M}\to \mathcal{F}_g$.
	\end{proof}
	
	Combined with \cite[Theorems 1.4 and 1.5]{FLV}, we get their generalization as follows. 
	
	\begin{cor}\label{cor:n(g)}
		Notation is as above. Assume that $2\leq g\leq 12$ and $g\neq 11$. Set the function
		$$m(g)=\begin{cases}
		3 & g=2, 4, \text{or } 5\\
		5 & g=3\\
		2 & g=6, 7, 8, 9, 10, \text{or } 12.
		\end{cases}$$
		Then for any $r, d, s\in \N$  such that $0\leq d^{2}(g-1)-rs+1\leq  m(g)$ and $\gcd(r, d, s)=1$, the relative moduli space $\mathcal{M}\to \mathcal{F}_g$ of $H$-stable sheaves with primitive Mukai vector $v =(r, dH, s)$ satisfies the Franchetta property.\qed
	\end{cor}

	\subsection{Franchetta property for certain families of hyper-K\"ahler varieties of O'Grady-10 type}

	\begin{thm}\label{thm:OG10}
		Given any $r, d, s\in \N$  such that $d^{2}(g-1)-rs=1$, let $\mathcal{M}\to \mathcal{F}_g$ be the relative moduli space of $H$-stable sheaves with Mukai vector $v =2(r, dH, s)$ and let $\widetilde{\mathcal{M}}\to \mathcal{M}$ be O'Grady's simultaneous crepant resolution.
		
		\begin{enumerate}[(i)]
			\item 	Assume that $\mathcal{S}^{5}_{/\mathcal{F}_g}\to \mathcal{F}_g$ satisfies the Franchetta property.
			Then  $\widetilde{\mathcal{M}}\to \mathcal{F}_g$ satisfies the Franchetta property.
			\item In case $g=3$ (quartic surface case), the family $\widetilde{\mathcal{M}}\to \mathcal{F}_3$ has the Franchetta property. In particular, the Beauville--Voisin conjecture is true for the very general element of $\widetilde{\mathcal{M}}\to \mathcal{F}_3$.
		\end{enumerate}
	\end{thm}
	
	\begin{proof} Statement $(\rom1)$ follows from Theorem \ref{thm:K3OG}, plus the observation that the construction of the split injection of that theorem can be performed relatively over $\FF_g$. Statement $(\rom2)$ follows from~$(\rom1)$, in view of the fact that $\Ss^{5}_{/\mathcal{F}_3}\to \mathcal{F}_3$ has the Franchetta property \cite[Theorem~1.5]{FLV}.
	\end{proof}

\subsection{The Chow ring of Ouchi's eightfolds}
\label{SS:Ouchi}
In \cite{Ou}, Ouchi has constructed certain hyper-K\"ahler eightfolds,  that we call Ouchi eightfolds, as moduli spaces of stable objects on the Kuznetsov component of a very general cubic fourfold \emph{containing} a plane. These form therefore a 19-dimensional family of projective hyper-K\"ahler eightfolds of $K3^{[4]}$-type.
We provide the following Beauville--Voisin type consequence of Theorem \ref{thm:FranchettaLLSS} for Ouchi eightfolds.

\begin{cor}\label{C:ouchi}
	Let $Y$ be a very general cubic fourfold containing a plane, and let $Z=Z(Y)$ be the associated Ouchi eightfold \cite{Ou}.
	Then the $\QQ$-subalgebra
	\[ \langle h, c_j({Z}), [Y]\rangle\ \ \subset\ \CH^\ast(Z) \]
	injects into cohomology, \emph{via} the cycle class map. Here, $h$ is the natural polarization and $Y\subset Z$ is the lagrangian embedding constructed in \cite{Ou}.
\end{cor}

\begin{proof} Referring to the notation of the proof of Theorem \ref{thm:FranchettaLLSS}, one considers the relative moduli space $\mathcal{M}\to B$, whose fiber over $b\in B$, denoted $M_{b}$, is $\mathcal{M}_{\sigma}(\mathcal{A}_{Y_{b}}, 2\lambda_{1}+\lambda_{2})$.  For $b\in B^{\circ}$, $M_{b}$ is isomorphic to the LLSS eightfold associated to $Y_{b}$ by \cite{LPZ18}\,; while for a very general point $b\in B\backslash B^{\circ}$, $M_{b}$ is isomorphic to the Ouchi eightfold associated to $Y_b$ by \cite[Example~32.6]{BLMNPS}. Moreover, the classes $h$, $c_{j}(M_{b})$ and $[Y_{b}]$ on Ouchi eightfolds are specializations of the corresponding classes on LLSS eightfolds.
	Thanks to Theorem \ref{thm:FranchettaLLSS}, the $\QQ$-subalgebra generated by $h$, $c_j$ and $[Y]$ (which are generically defined with respect to $B$) injects into cohomology. By specialization, the same holds true for Ouchi eightfolds.
\end{proof}

  \bibliographystyle{amsalpha}
	\bibliography{bib}

\end{document}